\documentclass[11pt]{amsart}
\usepackage{amscd,amsmath,amssymb,amsfonts,verbatim}
\usepackage{amscd, amssymb, mathrsfs, mathabx, tikz-cd, amsthm, thmtools}
\usepackage[cmtip, all]{xy}
\usepackage[letterpaper, left=3cm, right=3cm, top = 2.5cm, bottom = 2.5cm ]{geometry}
\usepackage{comment}
\usepackage{hyperref}
\usepackage{enumitem}

\usepackage{color}
\newtheorem{thm}{Theorem}[section]

\newtheorem{lem}[thm]{Lemma}
\newtheorem{cor}[thm]{Corollary}

\theoremstyle{definition}

\newtheorem{innercustomthm}{Theorem}
\newenvironment{customthm}[1]
{\renewcommand\theinnercustomthm{#1}\innercustomthm}
{\endinnercustomthm}

\newtheorem{defn}[thm]{Definition}
\theoremstyle{remark}
\newtheorem{remk}[thm]{Remark}
\newtheorem{remks}[thm]{Remarks}

\newtheorem{exm}[thm]{Example}
\newtheorem{exms}[thm]{Examples}
\newtheorem{notat}[thm]{Notation}
\newtheorem{constr}[thm]{Construction}
\numberwithin{equation}{section}

\newenvironment{Def}{\begin{defn}}%
	{\hfill$\square$\end{defn}}
\newenvironment{rem}{\begin{remk}}%
	{\hfill$\square$\end{remk}}
	{\hfill$\square$\end{remks}}
	{\hfill$\square$\end{exm}}
	{\hfill$\square$\end{exms}}
	{\hfill$\square$\end{notat}}

\newcommand{\thmref}{Theorem~\ref}

\newcommand{\corref}{Corollary~\ref}
\newcommand{\defref}{Definition~\ref}
\newcommand{\lemref}{Lemma~\ref}

\newcommand{\remref}{Remark~\ref}

\newcommand{\secref}{Section~\ref}

\newcommand{\QCoh}{{\operatorname{\rm QCoh}}}
\newcommand{\Fact}{{\operatorname{\rm Fact}}}

\newcommand{\Zs}{Z(s)}
\newcommand{\fBo}{\fB ^{\circ}}

\newcommand{\fB}{\mathfrak{B}}
\newcommand{\fM}{\mathfrak{M}}
\newcommand{\lan}{\langle}
\newcommand{\ran}{\rangle}

\newcommand{\kb }{{\beta}}
\newcommand{\h}{\mathrm{h}}

\newcommand{\LGQ}{LGQ}
\newcommand{\ti }{\times}

\newcommand{\vir}{\mathrm{vir}}
\newcommand{\ot }{\otimes}
\newcommand{\ke }{{\varepsilon }}
\newcommand{\ra }{\rightarrow}

\newcommand{\A}{{\mathbb A}}

\newcommand{\kl}{\lambda}

\newcommand{\Coh }{{\mathrm{Coh}}}

\newcommand{\Spec}{{\mathrm{Spec}}}

\newcommand{\Ker}{{\mathrm{Ker}}}
\newcommand{\ka }{{\alpha}}

\newcommand{\Sym}{{\mathrm{Sym}}}

\newcommand{\cO}{{\mathcal{O}}}
\newcommand{\cM}{{\mathcal{M}}}

\newcommand{\cF}{{\mathcal{F}}}

\newcommand{\cV}{{\mathcal{V}}}

\newcommand{\cK}{{\mathcal{K}}}

\newcommand{\fC}{{\mathfrak{C}}}

\newcommand{\sA}{{\mathcal A}}

\newcommand{\sC}{{\mathcal C}}

\newcommand{\sE}{{\mathcal E}}
\newcommand{\sF}{{\mathcal F}}

\newcommand{\sK}{{\mathcal K}}
\newcommand{\sL}{{\mathcal L}}

\newcommand{\sO}{{\mathcal O}}
\newcommand{\sP}{{\mathcal P}}

\newcommand{\sU}{{\mathcal U}}
\newcommand{\sV}{{\mathcal V}}

\newcommand{\sX}{{\mathcal X}}
\newcommand{\sY}{{\mathcal Y}}

\newcommand{\C}{{\mathbb C}}

\renewcommand{\P}{{\mathbb P}}

\newcommand{\inj}{\hookrightarrow}

\newcommand{\taut}{{\rm taut \ }}

\newcommand{\op}{\oplus}

\newcommand{\Tot}{{\operatorname{\rm Tot}}}

\newcommand{\DM}{{\operatorname{\mathcal{DM}}}}

\newcommand{\ds}{{/\kern-3pt/}}

\renewcommand{\dim}{\text{\rm dim}}

\newcommand{\tuborg}{\left\{\begin{array}{ll}}
	\newcommand{\sluttuborg}{\end{array}\right.}

\def\cO{\mathcal{O}}
\def\cF{\mathcal{F}}

\newcounter{elno}

\newcounter{elno-abc}   

\newcounter{elno-abc-prime}   

\usepackage{epigraph}

\begin{document}
	
	\title{Localization by $2$-periodic complexes and Virtual Structure Sheaves}
	\author[Oh]{Jeongseok Oh}
	\address{KIAS\\
		Seoul\\
		Korea}
	\email{jeongseok@kias.re.kr}
	
	\author[Sreedhar]{Bhamidi Sreedhar}
	\address{KIAS\\
		Seoul\\
		Korea}
	\email{sreedhar@kias.re.kr }
	
\begin{abstract}
	In \cite{KO1}, Kim and the first author proved a result comparing the  virtual fundamental classes of the moduli spaces of $\ke$-stable quasimaps and $\ke$-stable $LG$-quasimaps by studying localized Chern characters for $2$-periodic complexes.

	 In this paper, we study a $K$-theoretic analogue of the  localized Chern character map and show that for a Koszul $2$-periodic complex it coincides with the cosection localized Gysin map of Kiem-Li (\cite{KiemLi2}). As an application we compare  the  virtual structure sheaves of the moduli space  of $\ke$-stable quasimaps and $\ke$-stable $LG$-quasimaps.
	 
\end{abstract}
	
	\maketitle
	
	\setcounter{tocdepth}{1}
	\tableofcontents

	\section{Introduction}
	In \cite{KO1}, Kim and the first author  study  the localized Chern character map of Polishchuk-Vaintrob (\cite{PV:A}) and show that with respect to the tautological Koszul $2$-periodic complex  the localized Chern character map coincides with the cosection localization map of  Kiem-Li (\cite{KiemLi}). This allows the authors to prove a result comparing  the virtual fundamental classes  of the moduli space of stable quasimaps and the moduli space of stable $LG$-quasimaps. Chiodo in \cite{Chiodo} constructed Witten's top Chern class avoiding the use of bivariant intersection theory of \cite{PV:A} by using purely $K$-theoretic methods. In \cite{KiemLi2}, Kiem and Li constructed a cosection localization map for virtual structure sheaves. In this note we  study the $K$-theoretic analogue of \cite{KO1}. We generalize the results \cite{KO1} to $K$-theory and for a $2$-periodic complex of vector bundles with appropriate support we define a localization map (\secref{Sec:NotionsofSuppor} \& \defref{Defn:$h^Y_X$})  which plays the role of the localized Chern character map. This map is a tensor product followed by taking the alternating sum of the even and odd cohomology sheaves, which is precisely the map studied by Chiodo (see \cite[Lemma 5.3.4]{Chiodo}). We show that this map when applied to the tautological  Koszul $2$-periodic complex coincides with the cosection localized Gysin map of Kiem-Li (\cite{KiemLi2}). Now as in \cite{KO1}, this allows us to prove a result comparing the virtual structure sheaves  of the moduli space of stable quasimaps and the moduli space of  stable $LG$-quasimaps analogous to \cite[Theorem 3.2]{KO1}.

		The main result we prove in this text is as follows where the notations are explained in  \secref{subsec: SetupNotations}.  
	\begin{customthm}{\ref{Comp:Thm}}
		In the Grothendieck group of coherent sheaves on $Q^\ke_X:=Q^{\ke}_{g, k} ( Z(s) , d)$, we have
		\begin{align}\label{Vir:Eq2}[ \cO_{Q^{\ke}_X}^{\vir}] = (-1)^{\chi^{gen} ( R\pi_*\cV ^{\vee}_2 ) } \det R \pi_*(\cV_2 \ot \omega_\fC)^\vee  |_{Q_X^{\ke}} \ot  [ \cO_{LGQ',  dw_{LGQ'}}^{\vir} ]
		\end{align}
		where $\chi^{gen} ( \cV ^{\vee}_2 )$ is generic virtual rank  of  $R \pi _* \cV_2 ^{\vee}$ and $\det R \pi_*(\cV_2 \ot \omega_\fC)^\vee$ is the determinant (see \defref{defn:det}).
	\end{customthm}

A conceptually more intuitive explanation of \thmref{Comp:Thm} is as follows. Let $\sX$ be a Deligne-Mumford stack and let $\mathbb{E}=[A \ra B]$ be a perfect obstruction theory on $\sX$, where $A$ and $B$ are vector bundles.  Following  \cite[\S~2.1]{thomas}, we can define the determinant of $\mathbb{E}$ by, $$\sK_{\vir}  := (\det A)^\vee \ot \det B.$$  A square root of $\sK_{vir}$ is a line bundle $\sL$ (which need not exist) such that $\sL^2 = \sK_{\vir}$. It is usually denoted by  $\sK_{\vir}^{1/2}$ (\cite[\S~2.2]{thomas}) and is called the  square root of $\det \mathbb{E}$.

Let $LGQ'$ be the moduli space of $LG$-quasimaps and let $Q_X^{\ke}$ be the moduli space of $\ke$-stable quasimaps (see \secref{subsec: SetupNotations} and \defref{defn:det} for the precise notations). Let $\mathbb{E}_{LGQ'}$ and $\mathbb{E}_{Q_X^\ke}$ denote the canonical perfect obstruction theories defined on the respective moduli spaces (see \secref{subsec: SetupNotations}).  Let us assume that there exists a square root of $\det \mathbb{E}_{LGQ'}$ denoted by $\cK^{1/2}_{\LGQ',\vir}$ and let $\sK_{{vir}, Q_X^{\ke}}$ denote the determinant of $\mathbb{E}_{Q_X^\ke}$. Note that,  
\begin{align*}
\cK_{Q^\ke_X, \vir}& =  (\cK^{1/2}_{\LGQ', \vir} \ot \det R \pi_*(\cV_2 ^\vee))|_{Q_X^{\ke}}^{\ot 2}. 
\end{align*}
Thus, if $\cK^{1/2}_{\LGQ',\vir}$ exists  we have a canonical choice for the square root of  $\det \mathbb{E}_{Q_X^\ke}$ given by, 
$$\cK^{1/2}_{Q^\ke_X, \vir} := (\cK^{1/2}_{\LGQ', \vir} \ot \det R \pi_*(\cV_2 ^\vee))|_{Q_X^{\ke}}$$
Now  \thmref{Comp:Thm} can be restated as follows, 
	\begin{align} \label{re-main theorem}
[\cK^{1/2}_{Q^\ke_X, \vir} \ot \cO_{Q^{\ke}_X}^{\vir}] =  (-1)^{\chi^{gen} ( R\pi_*\cV ^{\vee}_2 ) } [ \cK^{1/2}_{\LGQ' ,\vir}|_{Q_X^{\ke}} \ot \cO_{LGQ',  dw_{LGQ'}}^{\vir} ], 
\end{align}

When $\ke = \infty$, note that the moduli spaces under consideration are quasi-projective, hence   \eqref{re-main theorem} is well defined  if $Q^\infty_{g,k}(Z(s), d)=\overline{\cM}_{g,k}(Z(s), d)$ possesses a square root of $\det \mathbb{E}_{Q_X^\ke}$  (i.e $\sK^{1/2}_{vir}$ exists).

	\subsection{Structure of the paper and relations to other work}\label{subsec:planofpaper}
	
	In \cite{KO1}, Kim and the first author obtained a comparison result of virtual classes of quasimap moduli spaces and $LG$-quasimap moduli spaces \cite[Theorem 3.2]{KO1} by using the localized Chern character of $2$-periodic complexes of vector bundles.  Using the ideas of {\it loc.cit.} in this text we  prove similar results for the virtual structure sheaves.
	
	 In \secref{LocK}, we fix the notations and assumptions that we shall follow in the rest of this text. Several steps in the proofs and constructions used are the $K$-theoretic analogues of \cite[Section 2]{PV:A}.  In \defref{Defn:$h^Y_X$} we define the map $\h^{\sY}_{\sX}(E_{\bullet})$  that would play the role that the localized Chern character plays in {\it loc.cit}. This definition is motivated by the map considered by  Chiodo in \cite[Lemma 5.3.4]{Chiodo}. In the rest of \secref{Sec:FunctorialityProperties} we study the functoriality properties of the map $\h^{\sY}_{\sX}(E_{\bullet})$ with respect to various functors in $K$-theory. As we shall require functoriality statements with respect to various derived functors, the natural setting for our statements is the derived category of Matrix Factorizations, we recall some preliminaries from \cite{MF} in the beginning of \secref{LocK} and in particular in \secref{Sec:NotionsofSuppor} we recall various notions of support and acyclicity for $2$-periodic complexes. 
	
The setup of  \secref{Kos:Comp} and \secref{section:com_sh} are similar to \cite[\S~2 , \S~3]{KO1}.	In \secref{Kos:Comp}, we recall the definition of the Koszul $2$-periodic complex associated to a section and a cosection of a vector bundle. In \cite{KiemLi2}, Kiem and Li construct a cosection-localized Gysin map to define a cosection-localized virtual structure sheaf. The goal of this section is to compare the map $\h^{\sY}_{\sX}(E_{\bullet})$, when $E_{\bullet}$ is the tautological Koszul $2$-periodic complex (\secref{taut}) to the construction of Kiem-Li (\cite[Theorem 4.1]{KiemLi2}). A similar comparison for the localized Chern character with the cosection localized Gysin map is \cite[Theorem 2.6]{KO1}. Building on this comparison result with the constructions Kiem-Li, in \secref{section:com_sh} we prove the main result (\thmref{Comp:Thm}). This section relies on the constructions of \cite[\S~3]{KO1} to aid the reader and fix notations we briefly summarize the results of \cite{KO1} in Section \ref{subsec: SetupNotations}  before proving  \thmref{Comp:Thm} in \secref{subsec: Proof of Main Theorem}.  
	\bigskip
	
	\noindent{\bf Acknowledgments}  
	The authors would like to thank Bumsig Kim for suggesting the problem to us and for indicating to us that \defref{Defn:$h^Y_X$} would be the $K$-theoretic counter part of the localized Chern character map. 
We would further  like to thank Bumsig Kim and David Favero for their precious advice and answering several doubts we had throughout the preparation of this text. J. Oh would like to thank Alexander Polishchuk and Richard Thomas for their interest in the result and useful advice. He also thanks Feng Qu and Mark Shoemaker for useful advice and discussions. 
	
	\section{Localization in $K$-theory}\label{LocK}
	\subsection{Assumptions and Notations} \label{Convention}

In this article we work over a  base  field $k$ of characteristic $0$. In \secref{section:com_sh} we shall further assume $k =\C$. All {\it schemes} and {\it algebraic stacks} are assumed to be separated and of finite type over $k$. In this article we shall mainly work with {\it Deligne-Mumford} stacks which we shall abbreviate as  $\DM$ stacks. For a $\DM$-stack $\sY$, let $\Coh (\sY)$ denote the abelian category of coherent sheaves and let $D^b(\sY)$ denote the bounded derived category of coherent sheaves on $\sY$. $G_0(\sY)$ shall denote the Grothendieck group of coherent sheaves on $\sY$, that is $G_0(\sY) := K_0(\Coh(\sY))$. Let $\sX$ be a closed substack of $\sY$ then we denote the category of coherent sheaves on $\sY$ supported on $\sX$ by $\Coh_{\sX} \sY$. It follows from d\'{e}vissage that the pushforward map induces a natural isomorphism $G_0(\Coh \sX) \xrightarrow{\cong} K_0(\Coh_{\sX} \sY)$. For a coherent sheaf $\sF$ we denote its class in $G_0$ by $[\sF]$.

\subsubsection{$2$-periodic complexes and the category of Matrix factorizations}
Let $\sY$ be a $\DM$-stack. 
A $2$-periodic complex of quasi-coherent sheaves on $\sY$ is a complex $G_{\bullet}$ as follows,
$$G_{\bullet} = [ \xymatrix{   G_{-}  \ar@<1ex>[r]^{d_{-} }  
	&  \ar@<1ex>[l]^{d_{+} } G_{+}  }] = ... \xrightarrow{d_+} G_{-} \xrightarrow{d_{-}} G_+ \xrightarrow{d_+} G_{-} \xrightarrow{d_{-}}  ... ,$$ 
where $G_-$ is in odd degree, $G_+$ is in even degree and $d_+ \circ d_- = d_- \circ d_+ = 0$.
$G_{\bullet}$ is said to be a $2$-periodic complex of vector bundles on $\sY$ if  we further require that $G_+$ and $G_-$ are vector bundles. 
For the proofs of some functoriality results it would be advantageous to work in a suitable derived category. As noted a $2$-periodic complex is an unbounded complex and hence some care is needed in framing the statements correctly. We briefly recall the notion of the derived category of Matrix factorizations(see \cite[\S~2]{MF} for a detailed discussion) and notions of support where our main focus is $2$-periodic complexes. 

Let $\sL$ be a line bundle on $\sY$ and let $w$ be a section of $\sL$. 
One can associate to the pair  $(\sY,w)$ the absolute derived category of factorizations $D^{abs}[\Fact(\sY,w)]$ and $D(\sY, w)$(see \cite[Definition 2.1.4]{MF}) which is the full subcategory of $D^{abs}[\Fact(\sY,w)]$ consisting of those factorizations which are isomorphic to factorizations with each component coherent. For a detailed discussion see \cite[\S~2]{MF}. We shall be interested in the particular case when $w=0$ where it is apparent from the definition that a $2$-periodic complex of coherent sheaves can be considered to be an object of $D(\sY,0)$. We  further note that following \cite[Definition 3.18]{BFK} there is a a folding map $\Upsilon: D^b(\QCoh~\sY) \to D^{abs}[(\Fact(\sY,0)]$ from the  bounded derived category of quasi-coherent sheaves to the absolute derived category of factorizations. The functor $\Upsilon$ restricted to the subcategory of coherent sheaves defines a functor $\Upsilon: D^b(\sY) \to D(\sY,0)$.

\begin{remk}\label{rem: Foldingforcoherentsheaf}
	For a coherent sheaf $\sF$ on a $\DM$-stack $\sY$ by $\Upsilon(\sF)$ we shall denote the folding of the complex with the coherent sheaf $\sF$ in degree zero and  all other terms zero. 
\end{remk}

\begin{defn}[Tensor Product of $2$-periodic complexes]\cite[Definition 2.2.1]{MF}
	Let $E_{\bullet}:=(E_+, E_-, d^E_+, d^E_-)$ and $F_{\bullet}:= (F_+, F_-, d^F_+, d^F_-)$ be $2$-periodic complexes of coherent sheaves. The  tensor product denoted by $E_{\bullet} \otimes F_{\bullet}$  is the $2$-periodic complex defined as follows, 
	\begin{align*}
	(E_{\bullet} \otimes F_{\bullet} )_+ & :=  E_+ \otimes F_+ \bigoplus E_- \otimes F_-\\
	(E_{\bullet} \otimes F_{\bullet} )_- & := E_+ \otimes F_- \bigoplus E_- \otimes F_+ 
	\end{align*}
	
	\begin{equation*}
	d^{E \otimes F}_+ = \begin{pmatrix}
	1_{E_+} \otimes d^{F}_+ & d^{E}_- \otimes 1_{F_-} \\
	-d^{E}_+ \otimes 1_{F_+} & 1_{E_-} \otimes d^{F}_-
	\end{pmatrix}
	\end{equation*}
	\begin{equation*}
	d^{E \otimes F}_- = \begin{pmatrix}
	1_{E_+} \otimes d^{F}_- & -d^{E}_- \otimes 1_{F_+} \\
	d^{E}_+ \otimes 1_{F_-} & 1_{E_-} \otimes d^{F}_+
	\end{pmatrix}
	\end{equation*}
\end{defn}

\begin{remk}
	Note that essentially by definition the folding map $\Upsilon$ is compatible with the tensor product.
\end{remk}
\subsubsection{Notions of Support:}\label{Sec:NotionsofSuppor}

Let us briefly outline what we want to do in this section and why these notions of support are important to us. As before let  $\sY$ be a $\DM$-stack and let $E_\bullet$ be a bounded complex of vector bundles on $\sY$ supported on a closed substack $\sX$. For any coherent sheaf $\sF$ on $\sY$, the complex 
$E_\bullet \otimes \sF$ is also supported on $\sX$. This observation need not be true if our complex $E_\bullet$ was unbounded. Note also by support for bounded complexes we mean cohomological support. A $2$-periodic complex is unbounded by definition and hence we need to consider different notions of support and the compatibility between them for the proofs of this section. Most of the statements are elementary  and we note them down here due to a lack of a suitable reference. Let us begin by briefly recalling the the construction of the absolute Derived Category of factorizations. If $F_\bullet$ and $G_\bullet$ are $2$-periodic complexes then a morphism $a: F_\bullet \to G_\bullet$ is just a morphism of complexes such that the obvious squares commute. Translating this to the language of factorizations this is precisely the abelian category $Z^0\Fact(\sY,0)$ (\cite[Remark 2.1.2]{MF}). Let $Ch^b(Z^0\Fact(\sY,0))$ denote category of bounded chain complexes of factorizations (i.e chain complexes of chain complexes). For any object $\sE_\bullet \in Ch^b(Z^0\Fact(\sY,0))$, one can associate its totalization denoted by $\Tot(\sE_\bullet) \in \Fact(\sY,0)$ (see \cite[2.1.3]{MF}). 
Let $Acyc(\sY,0)$ be the full saturated subcategory of $\Fact(\sY,0)$ which contains the totalizations of exact sequences in $Ch^b(Z^0\Fact(\sY,0))$.
Now the construction of the absolute derived category follows the  standard recipe of looking at the underlying homotopy category and taking the 
Verdier quotient by $Acyc(\sY,0)$. The derived category of factorizations is denoted by $D^{abs}[\Fact(\sY,0)]$. 

For the rest of this discussion let $\sX \inj \sY$ be a closed substack and let $\sU = (\sY \setminus \sX) \xrightarrow{j} \sY$ be the open complement.  As we shall be working with several notions of support we highlight the following definitions.

\begin{defn}\cite[Definition 2.2.4]{MF} \label{A:defnAbsolute}
	$F_\bullet \in D^{abs}[\Fact(\sY,0)]$ is called {\it absolutely acyclic} if and only if $0 \to F_\bullet$ is an isomorphism in $D^{abs}[Fact(\sY,0)]$. Further, we say $F_\bullet$ is {\it absolutely supported on $\sX$} if and only if $0 \to j^*F_\bullet$ is an  isomorphism  in  $D^{abs}[\Fact(\sU,0)]$.
\end{defn}

Unlike arbitrary factorizations for $2$-periodic complexes one can also talk about the cohomology sheaves as the composition of the differentials is zero. We can thus make sense of the following definition which in the bounded case coincides with  the natural notion of support.

\begin{defn}[Cohomological Acyclicity]\label{A:defn: cohomological}
	A $2$-periodic complex $F_\bullet$ is called {\it cohomologically acyclic} if and only if $H_\pm (F_\bullet) = 0$. Further, we say that $F_\bullet$ is {\it cohomologically supported on $\sX$} if and only if the cohomology sheaves $H_\pm (F_\bullet)$ are supported on $\sX$. 
\end{defn} 

\begin{remk}
	It is not true that notions of acyclicity as in \defref{A:defnAbsolute} and \defref{A:defn: cohomological} coincide for a general $2$-periodic complex. However, for a folding of a bounded complex both the notions are the same
\end{remk}

The next definition is the notion of locally contractible complexes following Polischuk-Vaintrob (see paragraph just preceding  \cite[Definition 3.13]{PV:Stacks}).

\begin{defn}[Locally contractible complex]
	A $2$-periodic complex $E_\bullet$ is said to be {\it locally contractible} if and only if there exists a smooth atlas $Y \xrightarrow{p} \sY$ such that $p^*E_\bullet$ is contractible on $Y$ (i.e. the identity map is homotopic to the zero map). Further, 
	we say that $F_\bullet$ is {\it locally contractible off $\sX$ }if and only if $j^*F_\bullet$ is locally contractible on $\sU$. 
\end{defn}

We now summarize some relations between  these notions of support in the following lemma.

\begin{lem}\label{A:lem:comparision}
	Let $\sY$ be a $\DM$ stack and let $E_\bullet$ be a $2$-periodic complex on $\sY$. Then the following hold
	\begin{enumerate}[label=(\roman*)]
		\item $E_\bullet$ is absolutely acyclic $\implies$ $E_\bullet$ is cohomologically acyclic.
		\item $E_\bullet$ is locally contractible $\implies$ $E_\bullet$ is cohomologically acyclic.  
		\item Let $E_\bullet $ be a $2$-periodic complex of vector bundles on $\sY$ which is locally contractible then $E_\bullet \otimes F_\bullet$ is locally contractible for any $2$-periodic complex $F_\bullet$ on $\sY$.
	\end{enumerate}
\end{lem}
\begin{proof}
	$(i)$ follows from \cite[Remark 2.2]{LP} and discussion preceding \cite[Definition 3]{Orlov2}. $(ii)$ is obvious from the definition of being locally contractible. 
	To show $(iii)$, let $E_\bullet$ be a $2$-periodic complex of vector bundles on $\sY$ which is locally contractible. As the question is local we can assume that $\sY$
	is a scheme and $E_\bullet$ is contractible on $\sY$. Let the homotopy between the identity and zero maps be  given by  $h_+: E_+ \to E_-$ and $h_-: E_- \to E_+$. 
	Further, let us assume that $E_\bullet= (E_+, E_-, d_+^E, d_-^E)$ and $F_\bullet = (F_+, F_-, d^F_+, d^F_-)$. Let $h'_\pm: (E_\bullet \otimes F_\bullet)_\pm \to (E_\bullet \otimes F_\bullet)_\mp$
	be the new homotopies defined as follows, 
	\begin{equation*}
	h'_+ = \begin{pmatrix}
	1_{E_+} \otimes d^{F}_+ & h_- \otimes 1_{F_-} \\
	-h_+ \otimes 1_{F_+} & 1_{E_-} \otimes d^{F}_-
	\end{pmatrix},
	\end{equation*}
	
	\begin{equation*}
	h'_- = \begin{pmatrix}
	1_{E_+} \otimes d^{F}_- & - h_- \otimes 1_{F_+} \\
	h_+ \otimes 1_{F_-} & 1_{E_-} \otimes d^{F}_+
	\end{pmatrix}.
	\end{equation*}
	Then $h'_\pm$ is the required homotopy which completes the proof of $(iii)$.
\end{proof}

\begin{remk}
We note that in \lemref{A:lem:comparision}(iii) the assumption that $E_\bullet$ is a complex of vector bundles is not necessary for the proof, it would be required for tensor to define an exact functor (see \defref{Defn:$h^Y_X$}).
\end{remk}
By restricting to the complement of the respective supports we have the following obvious corollary.
\begin{cor}\label{A:Cor-Comparision}
	Let $\sY$ be a $\DM$ stack and $\sX$ a closed substack. Let $E_\bullet$ be a $2$-periodic complex on $\sY$. Then the following hold
	\begin{enumerate}[label=(\roman*)]
		\item $E_\bullet$ is absolutely acyclic off $\sX$  $\implies$ $E_\bullet$ is cohomologically acyclic off $\sX$.
		\item $E_\bullet$ is locally contractible off $\sX$ $\implies$ $E_\bullet$ is cohomologically acyclic off $\sX$.  
		\item Let $E_\bullet $ be a $2$-periodic complex of vector bundles on $\sY$ which is locally contractible off $\sX$ then $E_\bullet \otimes F_\bullet$ is locally contractible off $\sX$ for any $2$-periodic complex $F_\bullet$ on $\sY$. In particular $E_\bullet \otimes F_\bullet$ is cohomologically acyclic off $\sX$.
	\end{enumerate}
\end{cor}
In the next remark we briefly explain how \cite[Proposition 2.2.8 \& Proposition 2.2.10]{MF} and their proofs can applied in our context. 

\begin{rem}[Proper Pushforward and the projection formula]\label{rem:Properpushforwardandprojectionformula}
 Let $g: \sY' \to \sY$ be a proper representable morphism. Then $Rg_*: D(\QCoh \sY') \to D(\QCoh\sY)$ has finite cohomological dimension. Now following {\cite[Definition 2.2.4]{MF} \& \cite[Theorem 1.17]{Nironi}}, 
 for any quasi-coherent sheaf we can find a finite resolution by $f_*$-acyclic sheaves which allows us to  conclude that there is a well defined functor $Rg_*: D^{abs}[\Fact(\sY',0)] \to D^{abs}[\Fact(\sY,0)]$.
 
 Let $E$ be a vector bundle on $\sY$, then the usual projection formula implies that $g^*E \otimes F$ is $f_*$-acyclic when $F$ is a $f_*$-acyclic sheaf. Arguing as in \cite[Proposition 2.2.10]{MF} we conclude that
   for a $2$-periodic complex of vector bundles $E_\bullet$ on $\sY$ the projection formula holds, that is 
$Rg_* (g^*E_\bullet \otimes F_\bullet) \cong E_\bullet \otimes Rg_* F_\bullet$. 

\end{rem}

\subsection{The Map $\h^{\sY}_{\sX}(E_{\bullet})$}
Let $\sY$ be a $\DM$-stack and let $\sX \xrightarrow{i} \sY$ be a closed substack. Let $E_{\bullet}$ be a $2$-periodic complex of vector bundles on $\sY$ such that 
$E_{\bullet}$ is locally contractible off $\sX$. In this section we would like to define a map $\h^{\sY}_{\sX}(E_{\bullet}): G_0(\sY) \to G_0(\sX)$ (see \defref{Defn:$h^Y_X$}) which would 
be the $K$-theoretic localized Chern character map for $2$-periodic complexes. We first need the following elementary lemma.
 
\begin{lem}\label{lem: welldefined}
	Let $$0 \ra E^1_\bullet \xrightarrow{f_{\bullet}} E^2_\bullet  \xrightarrow{g_{\bullet}} E^3_\bullet \ra 0,$$ be an exact sequence of $2$-periodic complexes of coherent sheaves on $\sY$ which are  cohomologically supported on $\sX$. Then 
	$$[H_+(E^2_\bullet)] - [H_-(E^2_\bullet)] =[ H_+(E^1_\bullet)] - [H_-(E^1_\bullet)] + [H_+(E^3_\bullet)] - [H_-(E^3_\bullet)]$$
in $K_0(\Coh_{\sX} \sY)$. 
\end{lem}
\begin{proof}
	
	Let us denote $d^i_{\pm}$ the the differentials $E^i_{\pm} \to E^i_{\mp}$. Let $K_\pm$ be the kernel of $\text{Im} d^2_\pm \ra \text{Im} d^3_\pm$.
 Consider the following commutative diagram,

	\begin{equation}\label{Diag:lem 1.2}
	\xymatrix{
		0 \ar[r]  & K_+ \ar[r]\ar[d] & \text{Im} d^2_+ \ar[d] \ar[r] & \text{Im} d^3_+ \ar[d] \ar[r] & 0 \\
		0 \ar[r]  & E^1_- \ar[r]\ar[d] & E^2_- \ar[d] \ar[r] & E^3_- \ar[d] \ar[r] & 0 \\
		0 \ar[r]  & K_- \ar[r] & \text{Im} d^2_-  \ar[r] & \text{Im} d^3_-  \ar[r] & 0.
	}
	\end{equation}
	
	\noindent	As $E^1_-$ is the kernel of the map $E^2_- \to E^3_-$, by the universal property of the kernel there exists a map $K_+ \to E^1_-$ such that the top left square commutes. In particular, we note that all the vertical maps in the top rectangle of \eqref{Diag:lem 1.2} are injective, therefore we can consider 
	$K_+ \subset E^1_-$. 
	Since $\text{Im} d^1_\pm \ra \text{Im} d^2_\pm$ are injective, $\text{Im} d^2_\pm \ra \text{Im} d^3_\pm$ are surjective and their compositions are zero morphisms, we obtain $\text{Im} d^1_\pm \subset K_\pm$ by the universal property.
	Hence, there exists a morphism $E^1_- \ra K_-$ which makes the left bottom square to be commutative.
	On the other hand, a composition of an inclusion $K_+ \subset E^1_- \subset E^2_-$ with $d^2_-$ is zero. Thus, the composition of the inclusion $K_+ \subset E^1_-$ with $d^1_-$ is zero. Again by the universal property, we obtain $K_+ \subset \Ker d^1_-$ (and similarly we have $K_- \subset \Ker d^1_+$).
	It implies that the composition of the  left vertical map in the above diagram is zero.

	By taking the quotient of the middle exact sequence by the top exact sequence in diagram \eqref{Diag:lem 1.2} we obtain  the following commutative diagram. 
	\begin{equation}\label{Diag:lem 1.2-2}
	\xymatrix{
		0 \ar[r]  & E^1_-/ K_+  \ar[r]\ar[d] & E^2_-/\text{Im} d^2_+ \ar[d] \ar[r] & E^3_-/\text{Im} d^3_+ \ar[d] \ar[r] & 0 \\
		0 \ar[r]  & K_- \ar[r] & \text{Im} d^2_-  \ar[r] & \text{Im} d^3_-  \ar[r] & 0.
	}
	\end{equation}
	
	\noindent From the snake lemma applied to diagram \eqref{Diag:lem 1.2-2} we get the following exact sequence
	
	\begin{equation}\label{eqn: 1}
	0 \ra \frac{\Ker d^1_-}{K_+} \ra H_-(E^2_\bullet) \ra H_-(E^3_\bullet) \ra \frac{K_-}{\text{Im}d^1_-} \ra 0
	\end{equation}
	Further, by $\text{Im} d^1_\pm \subset K_\pm \subset \Ker d^1_\mp$, we have the following exact sequence
	\begin{equation}\label{eqn:2}
	0 \ra \frac{K_-}{\text{Im}d^1_-} \ra \frac{\text{Ker}d^1_+}{\text{Im}d^1_-} \ra \frac{\text{Ker}d^1_+}{K_-} \ra 0
	\end{equation}
	
	Using \eqref{eqn:2} in \eqref{eqn: 1} we get the following exact sequence
	\begin{equation}\label{eqn:3}
	0 \ra \frac{\Ker d^1_-}{K_+} \ra H_-(E^2_\bullet) \ra H_-(E^3_\bullet) \ra H_+(E^1_\bullet) \ra \frac{\Ker d^1_+}{K_-} \ra 0.
	\end{equation}

	\noindent Similarly, we have 
	\begin{equation}\label{eqn:4}
	0 \ra \frac{\Ker d^1_+}{K_-} \ra H_+(E^2_\bullet) \ra H_+(E^3_\bullet) \ra H_-(E^1_\bullet) \ra \frac{\Ker d^1_-}{K_+} \ra 0.
	\end{equation}
	
	\noindent It follows from \eqref{eqn:3} and \eqref{eqn:4} that we have the following equality, 
	
	\begin{equation}\label{eqn:5}
[	H_+(E^2_\bullet)] - [H_-(E^2_\bullet)] = [H_+(E^1_\bullet)] - [H_-(E^1_\bullet)] +[ H_+(E^3_\bullet)] - [H_-(E^3_\bullet)],
	\end{equation}
	in $K_0(\Coh_{\sX} \sY)$.
	
\end{proof}

	Let $$0 \to \mathcal{F}_1 \to \mathcal{F}_2 \to \mathcal{F}_3 \to 0,$$
	be an exact sequence in $\Coh(\sY)$ and let $E_{\bullet}$ be a $2$-periodic complex of vector bundles on $\sY$ such that $E_{\bullet}$ is locally contractible off $\sX$. Then we have,   
	\begin{equation}
0 \to E_{\bullet} \ot \Upsilon(\sF_1) \to E_{\bullet} \ot \Upsilon(\sF_2) \to E_{\bullet} \ot \Upsilon(\sF_3) \to 0,
	\end{equation}
is an exact sequence of $2$-periodic complexes on $\sY$. Further, from \corref{A:Cor-Comparision}(iii) it follows that  $E_{\bullet} \ot \Upsilon(\sF)$ is cohomologically supported on $\sX$. Now \lemref{lem: welldefined}  allows us to make the following definition.

\begin{Def}\label{Defn:$h^Y_X$}
	Let $E_{\bullet}$ be a $2$-periodic complex of vector bundles on $\sY$ such that $E_{\bullet}$ is locally contractible off $\sX$. We define  $$\h ^{\sY}  _{\sX}  (E_\bullet): G_0(\sY) \rightarrow G_0(\sX)$$
	by sending a coherent sheaf $\mathcal{G}$ on $\sY$ to
	\begin{align*}
	\h^{\sY}  _{\sX}  (E_\bullet) ({\mathcal{G}})  : =   [H_+( E_\bullet \otimes {\Upsilon(\mathcal{G})})] - [H_-(E_\bullet \otimes {\Upsilon(\mathcal{G})})].
	\end{align*}
It follows from \lemref{lem: welldefined} that the map $\h^{\sY}  _{\sX}  (E_\bullet)$ extends  linearly to define a group homomorphism  $G_0(\sY) \to G_0(\sX)$, where we have implicitly used the isomorphism  $G_0(\sX) \xrightarrow{\cong} K_0(\Coh_{\sX} \sY)$. 
	
\end{Def} 

\subsection{Some Functorialities}\label{Sec:FunctorialityProperties}
In the rest of this section we prove some functoriality properties of the map $\h^\sY_\sX(E_\bullet)$ defined in \defref{Defn:$h^Y_X$}. If $E_\bullet$ is a folding of a bounded complex of vector bundles then there is nothing to prove. For our application we 
need these properties when $E_\bullet$ is not a folding of a bounded complex, in particular for Koszul $2$-periodic complexes (\defref{defn:Koszul2periodicComplex}). The strategy of proof is to reduce the arguments to the folding complexes  via Construction \ref{construction: Homotopyinvarianceconstruction} 
and a $\A^1$-homotopy invariance argument. We first make the following definition which is a  Gysin pull-back for $2$-periodic complexes.
\begin{defn}[$\lambda^!$ of a $2$-periodic complex]\label{defn:2periodicgysin}
Let $\widetilde{\sY}$ be a stack over $\mathbb{A}^1_k$ with a  closed substack $\sX \times_k \mathbb{A}^1_k \subset \widetilde{\sY}$ such that the composition $\sX \times_k \mathbb{A}^1_k \subset \widetilde{\sY} \ra \mathbb{A}^1_k$ coincides with the projection.
Let $\sE_\bullet$ be a $2$-periodic complex of coherent sheaves on $\widetilde{\sY}$ such that $\sE_\bullet$ is cohomologically supported on ${\sX \times_k \mathbb{A}^1_k}$.
Consider the following cartesian square,
\begin{equation}
	\xymatrix{
	\widetilde{\sY}|_\lambda \ar[r] \ar[d] & \widetilde{\sY} \ar[d]^p \\
	\Spec~ k \ar[r]_\lambda &            \A^1_k,
}
\end{equation}
where the bottom horizontal map is the regular embedding of the closed point $\lambda$. We have the natural resolution by vector bundles of $\lambda_*\sO_{\Spec k}$ given by $Q_{\bullet} := 0 \to k[t] \xrightarrow{t-\lambda} k[t] \to 0$.  Analogous to the refined Gysin map for $K$-theory (see \cite[\S~2]{YPLee}), let $P_{\bullet} := \Upsilon(Q_{\bullet})$ then we define  $\lambda^!(\sE_{\bullet}) := \sE_{\bullet} \otimes p^*P_{\bullet}$. Note that $\lambda^!(\sE_{\bullet})$ is  cohomologically supported on  $\sX \times {\lambda}$.
\end{defn}

	 The notation we use in \defref{defn:2periodicgysin} for the $2$-periodic Gysin map might appear to be slightly confusing to the reader. $\lambda^!$ is generally reserved for the Gysin map defined at the level of $K$-theory. The reason for this abuse of notation would be apparent from the following lemma.

\begin{lem}\label{lem:gysin}
    Let ${\sY} \times_k \A^1_k$ be a $\DM$-stack and let $\sX \times_k \A^1_k \inj \sY \times_k \A^1_k$ be a closed substack. Let $\sE_{\bullet}$ be a  $2$-periodic complex of coherent sheaves on $\sY \times_k \A^1_k$ cohomologically supported on $\sX \times_k \A^1_k$. Let $0^!(\sE_{\bullet})$ be defined as in \defref{defn:2periodicgysin}. Then, 
    \begin{equation}
        0^!([H_+(\sE_{\bullet})] - [H_-(\sE_{\bullet})]) = [H_+(0^!\sE_{\bullet})] - [H_-(0^!\sE_{\bullet})],
    \end{equation}
    in $G_0(\sX \times \Spec~ k)$ where $0^!([H_+(\sE_{\bullet})] - [H_-(\sE_{\bullet})])$ denotes the usual Gysin pull-back in $K$-theory (see \cite[\S~2]{YPLee}).
\end{lem}
\begin{proof}
The proof follows from the convergence of the spectral sequence defined by the following  double complex,
\[
\xymatrix{
 & \vdots & \vdots & \\
0 \ar[r] & \sE_+ \ar[r]^-t \ar[u] & \sE_+ \ar[r] \ar[u] & 0 \\
0 \ar[r] & \sE_- \ar[r]^-t \ar[u]^-{d_-} & \sE_- \ar[r] \ar[u]^{d_-} & 0 \\
0 \ar[r] & \sE_+ \ar[r]^-t \ar[u]^-{d_+} & \sE_+ \ar[r] \ar[u]^{d_+} & 0 \\
 & \vdots \ar[u] & \vdots \ar[u] & .\\
}
\]
Note that $0^!\sE_\bullet$ is a total complex of the above double complex.
\end{proof}

\begin{constr}\label{construction: Homotopyinvarianceconstruction}
	Following \cite[Proposition 2.5.8]{MF},  we consider the following construction. As before let $\sY$ be a $\DM$-stack and let $\sX$ be a closed substack. 
	Let $A_{\bullet}$ be a $2$-periodic complex of coherent sheaves on $\sY$ such that $A_{\bullet}$ is cohomologically supported on ${\sX}$,
	\begin{equation}
		\ldots \to A_+ \xrightarrow{d_+} A_- \xrightarrow{d_-} A_+ \xrightarrow{d_+} A_- \to \ldots.
	\end{equation}
Let $I_{\bullet}$ be the folding of the bounded complex $0 \to Img~d_- \xrightarrow{0} Img~ d_+ \to 0$ and similarly let $K_{\bullet}$ denote the folding of the complex $0 \to \Ker~ d_- \xrightarrow{0} \Ker~ d_+ \rightarrow 0$. The inclusion maps $Img \ d_{\mp} \inj \Ker d_{\pm}$ induces a map $I_{\bullet}[1] \to K_{\bullet}$. We thus have the following triangle in $D(\sY,0)$, 
\begin{equation}
I_{\bullet}[1] \to K_{\bullet} \to C_{\bullet} \xrightarrow{+1},
\end{equation}
where $C_{\bullet}$ is the mapping cone. Note that by construction $C_{\bullet}$ is an object of $D(\sY,0)$ which is cohomologically supported on $\sX$.
Let $H_{\bullet}$ be the $2$-periodic complex given by $H_+(A_{\bullet})$ in even degree and $H_-(A_{\bullet})$ in odd degree where both the odd and even differentials are zero, that is $H_{\bullet} := \Upsilon(H_+(A_{\bullet}) \oplus (\Upsilon(H_-(A_{\bullet}))[1])$. 
The two exact sequences, 
$$0 \to Img~d_{\mp} \to \Ker d_{\pm} \to H_{\pm}(A_{\bullet}) \to 0,$$
induce an isomorphism, 
\begin{equation}\label{eq:constrfoldingandcohomology}
C_{\bullet} \xrightarrow{\cong} H_{\bullet},
\end{equation}
 in $D(\sY,0)$. 
 There is a natural map $I_\bullet[1] \xrightarrow{i} A_{\bullet}$ induced by the inclusion $I_{\pm} \xrightarrow{i} A_{\mp}$. Let $B_{\bullet} := Cone(I_{\bullet}[1] \xrightarrow{i} A_{\bullet})$.  Let $\A^1_k = \Spec~k[t]$ and let $p: \sY \times_k \A^1_k \to \sY$ be the projection to $\sY$. Let $B[t]_{\bullet} := p^*(B_{\bullet})$, in particular we have $B[t]_{+} = A_+[t] \oplus I_+[t]$ and $B[t]_- = A_-[t] \oplus I_-[t]$.
Let us define the $f_{\bullet}: B[t]_{\bullet} \to I[t]_{\bullet}$ as follows, 
\begin{align*}
f_+: = d^{A[t]}_+ - t  id_{I_+[t]} &: A_{+}[t] \oplus I_+[t] \ra I_+[t], \\
f_- : = d^{A[t]}_- - t  id_{I_-[t]} &: A_{-}[t] \oplus I_-[t] \ra I_-[t].
\end{align*}

Note that $f_{\bullet}$ is a surjective map of complexes and let us denote by $\mathcal{A}_{\bullet}$ the kernel of $f_{\bullet}$. 
We have the following triangle in $D(\sY \times \A^1_k,0)$, 
\begin{equation}\label{constr:eq1}
	\mathcal{A}_{\bullet} \xrightarrow{t} \mathcal{A}_{\bullet} \xrightarrow{} \text{Cone}(t) \xrightarrow{+1},
\end{equation} 
where $\text{Cone}(t)$ is the cone of the map $\sA_{\bullet} \xrightarrow{t} \sA_{\bullet}$ which coincides with $0^!\sA_{\bullet}$ as in \defref{defn:2periodicgysin}. Further, by construction we have the following exact sequence of $2$-periodic complexes, 
\begin{equation}\label{constr:eq2}
	0 \to \mathcal{A}_{\bullet} \xrightarrow{t} \mathcal{A}_{\bullet} \to \mathcal{A}_{\bullet}/t\mathcal{A}_{\bullet} \to 0. 
\end{equation}
Combining \eqref{constr:eq1} and \eqref{constr:eq2} we have an isomorphism $0^! \mathcal{A}_{\bullet} \to \mathcal{A}/t\mathcal{A}$ in $D(\sY \times_k \A^1_k)$ where both the complexes are cohomologically supported on $\sX \times 0$. Let $j_0: \sY \times \{0\} \to \sY \times_k \A^1_k$ be the natural closed 
immersion. It follows from \eqref{constr:eq2} and the definition of $\mathcal{A}_\bullet$ that, 
\begin{equation}\label{constr:gysin0}
j_{0*} C_{\bullet} \cong 0^! \mathcal{A}_\bullet.
\end{equation}
  Repeating the above argument but this time for the inclusion $\Spec~k[t]/[t-1] \to \A^1_k$, that is corresponding 
to the closed point $1$ and with $j_1: \sY \times \{1\} \to \sY \times \A^1_k$ we can on the other hand  conclude  that, 
\begin{equation}\label{constr:gysin1}
 j_{1*} A_{\bullet} \cong 1^! \mathcal{A}_{\bullet}.
\end{equation}
 
  The  complex  $\sA_\bullet$  on $\sY \times_k \A^1_k$ gives us a deformation of our original complex $A_\bullet$ to a folding complex $C_\bullet$.
\end{constr}

\begin{lem}\label{Ring:hom}
	Let $\sY$ be a $\DM$-stack and let $\sX_1$ and $\sX_2$ be two closed substacks of $\sY$. Let   $E^1_{\bullet}$ and $E^2_{\bullet}$ be $2$-periodic complexes on $\sY$ such that $E^i_{\bullet}$ is locally contractible off $\sX_i$. Let $\sX = \sX_1 \cap \sX_2$ and let $i_1: \sX_1 \inj \sY$ denote the 
	closed immersion. Then we have, 

		\begin{equation}
		\h^{\sY}_{\sX}( E^1_\bullet \ot E^2_\bullet)(-)  = \h^{\sX_1}_{\sX}(i_1^*E^2_\bullet) ( \h^{\sY}_{\sX_1}(E^1_{\bullet})(-)).
		\end{equation}    
\end{lem}
\begin{proof}
Let $\mathcal{G}$ be a coherent sheaf on $\sY$ and let  $G_\bullet := E^1_{\bullet} \otimes \Upsilon(\mathcal{G})$. Now substituting  $A_{\bullet} = G_{\bullet}$ in Construction \ref{construction: Homotopyinvarianceconstruction} and keeping all other  notation the same we 
can construct $\mathcal{A}_{\bullet}$ on $\sY \times_k \A^1_k$ such that, $0^! \mathcal{A}_{\bullet} = j_{0*}C_{\bullet}$ and $1^!(\sA_{\bullet}) = j_{1*} G_{\bullet}$ and further the natural map $C_{\bullet} \to H_{\bullet}$ is an isomorphism where both are cohomologically supported on $\sX_1$. Therefore, we can 
conclude that $E^2_{\bullet} \otimes C_{\bullet} \to E^2_{\bullet} \otimes H_{\bullet}$ is an isomorphism in 
$D(\sY,0)$ where now both are cohomologically supported on $\sX$. Consider the  following  diagram, 
\begin{equation}\label{eq:RingHom1}
\xymatrix
{
\Coh \sX_1 \ar[r]^{i_{1*}} \ar[d]_{\Upsilon} & \Coh_{\sX_1} \sY \ar[r]^{\Upsilon} & D(\sY,0)_{\sX_1} \ar[d]^{\otimes E^2_{\bullet}} \\
D(\sX_1,0) \ar[r]_{\otimes( i_1^*E^2_{\bullet})}             & D(\sX_1,0)_\sX \ar[r]_{i_{1*}}                  & D(\sY,0)_{\sX}.
}
\end{equation}
Firstly note that $\Upsilon$ commutes with $i_{1*}$, then it follows from the projection formula \cite[Proposition 2.2.10]{MF} (see \remref{rem:Properpushforwardandprojectionformula}) that \eqref{eq:RingHom1} commutes. 

Now from the definition we have, 

\begin{align*}
\h^{\sX_1}_{\sX}(i_1^*E^2_{\bullet})(\h^{\sY}_{\sX_1}(E^1_{\bullet})(\mathcal{G})) &= \h^{\sX_1}_{\sX}(i_1^*E^2_{\bullet})(({i_{1*}})^{-1}([H_+(E^1_{\bullet} \otimes \Upsilon(\mathcal{G}))] - [H_-(E^1_{\bullet} \otimes \Upsilon(\mathcal{G})]))\\
                                                                                                             &= [H_+(E^2_{\bullet} \otimes H_{\bullet})] - [H_-(E^2_{\bullet} \otimes H_{\bullet})], 
\end{align*}
where the second equality follows from the commutativity of \eqref{eq:RingHom1}. Thus it is enough to show that, 
\begin{equation}\label{eq:RingHomSufficient}
	\h^{\sY}_{\sX} (E^1_{\bullet} \otimes E^2_{\bullet})(\mathcal{G}) = [H_+(E^2_{\bullet} \otimes C_{\bullet})] - [H_-(E^2_{\bullet} \otimes C_{\bullet})]. 
\end{equation} 
 Recall, we have the following diagram, 
 \begin{equation}
\xymatrix { \sY \ar[r]^-{j_i}  \ar[rd] & \sY \times_k \A^1_k  \ar[d]^p \\
	                                   &   \sY       }
 \end{equation}
 where $p$ is the projection to $\sY$.\\
 
Consider the following, 

\begin{align}\label{eq:lemRingHom1}
& j_{1*} ([H_+(E^2_{\bullet} \otimes E^1_{\bullet} \otimes \Upsilon(\mathcal{G}))] -[H_-(E^2_{\bullet} \otimes E^1_{\bullet} \otimes \Upsilon(\mathcal{G}))]) \\ \nonumber
&=  [H_+(j_{1*}(E^2_{\bullet} \otimes G_{\bullet}))] - [H_-(j_{1*}(E^2_{\bullet} \otimes G_{\bullet})) ] \\ \nonumber
&=   [H_+(p^*E^2_{\bullet} \otimes j_{1*}G_{\bullet})] - [H_-(p^*E^2_{\bullet} \otimes j_{1*}G_{\bullet}) ]\\ \nonumber
&=  [H_+(p^*E^2_{\bullet} \otimes1^! \mathcal{A}_{\bullet})] - [H_-(p^*E^2_{\bullet} \otimes 1^! \mathcal{A}_{\bullet}) ] \\ \nonumber
& = 1^!([H_+(p^*E^2_{\bullet} \otimes \mathcal{A}_{\bullet})] - [H_-(p^*E^2_{\bullet} \otimes \mathcal{A}_{\bullet}) ]). \nonumber
\end{align} 
The second equality follows from the projection formula  \cite[Proposition 2.2.10]{MF}(see \remref{rem:Properpushforwardandprojectionformula}) and the third equality follows from \eqref{constr:gysin1}. The last equality follows from \lemref{lem:gysin}. Arguing similarly, but this time applying \eqref{constr:gysin0} we have the following,

\begin{align}\label{eq:lemRingHom2}
&	j_{0*}([H_+(E^2_{\bullet} \otimes C_{\bullet})] - [H_+(E^2_{\bullet} \otimes C_{\bullet})]) \\ \nonumber
~~&=[H_+(p^*E^2_{\bullet} \otimes j_{0*}C_{\bullet})] - [H_+(p^*E^2_{\bullet} \otimes j_{0*}C_{\bullet})]\\ \nonumber
~~&= [H_+(p^*E^2_{\bullet} \otimes 0^!C_{\bullet})] - [H_+(p^*E^2_{\bullet} \otimes 0^!C_{\bullet})] \\ \nonumber
~~&= 0^!([H_+(p^*E^2_{\bullet} \otimes \sA_{\bullet})] - [H_+(p^*E^2_{\bullet} \otimes \sA_{\bullet} )]). \nonumber
\end{align} 
Combining \eqref{eq:lemRingHom1}, \eqref{eq:lemRingHom2}  and the fact that $0^! = 1^! : G_0(\sX \times_k \A^1_k) \to G_0(\sX)$ we conclude that 
\eqref{eq:RingHomSufficient} holds which completes the proof.
\end{proof}
\begin{lem}\label{lem:functoriality:pullback}
	Let $Z' \xrightarrow{g_Z} Z$ be a regular immersion  of quasi-projective schemes. Let $\sY$ be a $\DM$-stack, with $\sX \xrightarrow{i_{\sX}} \sY$ a closed substack. Let $p_{\sY} : \sY \to Z$ be a morphism. Consider the following diagram where each square is cartesian, 
		\begin{equation}\label{eqn: functorialitypullback}
	\xymatrix
	{
		\sX' \ar[r]^{i_{\sX'}} \ar[d]_{g_{\sX}} & \sY' \ar[r]^{p_{\sY'}} \ar[d]_{g_{\sY}} & Z' \ar[d]^{g_Z} \\
		\sX \ar[r]_{i_{\sX}}           & \sY  \ar[r]_{p_{\sY}}          & Z.
	}
	\end{equation}
	Let $E_{\bullet}$ be a $2$-periodic complex of vector bundles such that $E_{\bullet}$ is locally contractible off $\sX$, then for $\mathcal{G} \in G_0(\sY)$, 
	\begin{equation}
	g_{\sX}^!\h^{\sY}_{\sX}(E_{\bullet})(\mathcal{G}) = \h^{\sY'}_{\sX'}(g_{\sY}^*(E_{\bullet}))(g_{\sY}^!(\mathcal{G})),
	\end{equation}
	where $g_{\sX}^!: G_0(\sX) \to G_0(\sX')$ and $g_{\sY}^!: G_0(\sY) \to G_0(\sY')$ denote the {\it refined Gysin pull back} induced by the regular immersion $Z' \to Z$. 
	
\end{lem}

\begin{proof}
As $g_Z$ is a regular closed embedding and $Z'$ is quasi-projective  there exists a  resolution of $g_{Z*} \mathcal{O}_{Z'}$ by a complex of locally free sheaves which we denote by $P_{\bullet}$. Further, note that $\Upsilon$ is a monoidal functor with respect to the tensor product that is in particular the following diagram commutes, 
\begin{equation}
	\xymatrix
	{
D^b(\sY) \ar[r]^{\otimes p^*_{\sY}P_{\bullet}} \ar[d]_{\Upsilon} & D^b(\sY,0)_{\sY'} \ar[d]^{\Upsilon} \\
D(\sY,0)                                                     \ar[r]_{\otimes \Upsilon( p^*_{\sY}P_{\bullet})}                  & D(\sY,0)_{\sY'}.
	}
\end{equation}
The proof now follows from the multiplicative property \lemref{Ring:hom}.
\end{proof}

We restate  \lemref{lem:functoriality:pullback} in the set-up of \defref{defn:2periodicgysin} which would be in the form that we would want to apply it. We recall the notations for the convenience of the reader. 

\begin{lem}
	\label{homInv}
	Let $\sA_\bullet$ be a $2$-periodic complex of vector bundles on $\widetilde{\sY}$ and let $\widetilde{\mathcal{G}} \in G_0(\widetilde{\sY})$.
	Assume that the followings are satisfied:
	\begin{enumerate}
		\item There is a morphism $\widetilde{\sY} \rightarrow \mathbb{A}^1_k$. 
		\item There is a closed substack $\sX \times_k \mathbb{A}^1_k \inj \widetilde{\sY}$, such that the composition $\sX \times_k \mathbb{A}^1_k \inj \widetilde{\sY} \ra \mathbb{A}^1_k$ is a projection
		and $\sA_\bullet$ is  locally contractible off $\sX \times \A^1_k$.
		\item Let $p_t: \widetilde{\sY}|_{t} \to \widetilde{\sY}$ be the natural immersion. Assume that  $\h^{\widetilde{\sY}|_t}_{\sX}   (p_t^*\sA_\bullet)(t^! \widetilde{\mathcal{G}})=0$ for { any} $t \in \mathbb{A}^1_k \backslash \{0\}$.
	\end{enumerate}
	Then $\h^{\widetilde{\sY} |_0}_{\sX}  (p_0^*\sA_\bullet)(0^!\widetilde{\mathcal{G}})=0$. Further, when $\widetilde{\sY}= \sY \times \A^1_k$ then $\h^{\sY}_{\sX}(p_0^*\sA_{\bullet}) = h^{\sY}_{\sX}(p_1^*\sA_\bullet): G_{0}(\sY) \to G_0(\sX)$.
\end{lem}

\begin{proof}
	It follows from Lemma \ref{lem:functoriality:pullback} and $\mathbb{A}^1$-homotopy invariance $G_0(\sX \times_k \mathbb{A}^1_k) \cong G_0(\sX)$ (\cite[Proposition 3.3(4)]{Toen}).  
\end{proof}

\begin{lem}\label{lem:functoriality:Proper}
	Let $\sY$ and $\sY'$ be $\DM$-stacks  and let $g_Y: \sY' \to \sY$ be a proper representable morphism. Consider the following cartesian diagram,
	\begin{equation}\label{eqn1:lem:functoriality:Proper}
	\xymatrix{\sX' \ar[d]_{g_\sX} \ar[r]^{i'} & \sY' \ar[d]^{g_\sY} \\
		\sX \ar[r]_{i} & \sY, 	
	}         
	\end{equation}
	where $\sX$ and $\sX'$ are closed substacks of $\sY$ and $\sY'$ respectively. Let $E_{\bullet}$ be a $2$-periodic complex of vector bundles such that $E_{\bullet}$ is locally contractible off $\sX$. Then for every $\mathcal{G} \in G_0(\sY')$, 
	
	$$\h^{\sY}_{\sX}(E_{\bullet})(g_{\sY*} \mathcal{G}) = g_{\sX*}(\h^{\sY'}_{\sX'}(g^*_\sY E_{\bullet})(\mathcal{G})) \text{ in }  G_0(\sX).$$
\end{lem}

\begin{proof}
As $g_\sY$ is representable proper morphism it follows that $Rg_{\sY*}$ has finite cohomological dimension and hence it follows from  \remref{rem:Properpushforwardandprojectionformula} that $Rg_{\sY*} : D^{abs}[\Fact(\sY,0)] \to D^{abs}[\Fact(\sY,0)]$ is well defined. 
Let $\sU = \sY \setminus \sX$ and let us consider the following diagram, 
	\begin{equation}\label{eqn1:lem:functoriality:Proper2}
\xymatrix{\sX' \ar[d]_{g_\sX} \ar[r]^{i'} & \sY' \ar[d]^{g_\sY}  & \ar[l]_{j'} \sU' := \sY' \setminus \sX' \ar[d]^{g_{\sU}}\\
	\sX \ar[r]_{i} & \sY 	 & \ar[l]^{j} \sU := \sY \setminus \sX
}         
\end{equation}
where each square is cartesian. As $E_\bullet $ is locally contractible off $\sX$ it follows that $g_\sY^* E_\bullet$ is locally contractible off $\sX'$. Let $\mathcal{G}$ be a coherent sheaf on $\sY'$ and let us consider the $2$-periodic complex $G_\bullet:= g_\sY^*E_\bullet \otimes \Upsilon(\mathcal{G})$. 
From \corref{A:Cor-Comparision} it follows that $G_\bullet$ is cohomologically supported on $\sX'$. Further, from the projection formula we have,
\begin{equation}\label{eqn:projformulaedit}
Rg_{\sY*} G_\bullet \cong E_\bullet \otimes Rg_{\sY*} \Upsilon(\mathcal{G}) \cong E_\bullet \otimes \Upsilon(Rg_{\sY*} \mathcal{G}).
\end{equation}
 Now note that, $Rg_{\sY*}\mathcal{G}$  is an object of the bounded derived category of quasi-coherent sheaves with coherent cohomology, that is it is an object of $D^b_{coh}(\QCoh \sY)$. Therefore, $\Upsilon(Rg_{\sY*}\mathcal{G})$ is an object of $D(\sY,0)$. In particular it follows from \eqref{eqn:projformulaedit} that $Rg_{\sY*} G_\bullet$ belongs to $D(\sY,0)$.

The strategy of the proof now is to reduce the argument to folding complexes via Construction \ref{construction: Homotopyinvarianceconstruction}. First note that from \eqref{eqn1:lem:functoriality:Proper}   the following diagram commutes, 
\begin{equation}
\xymatrix{ 
	G_0(\sX') \ar[r]^-{i'_*} \ar[d]_{g_{\sX*}} & K_0(\Coh_{\sX'} \sY') \ar[d]^{g_{\sY*}}  \\
	G_0(\sX)  \ar[r]_-{i_*}                           & K_0(\Coh_{\sX} \sY) 
}
\end{equation} 
where the top and bottom horizontal maps are isomorphisms by d\'{e}vissage. Thus, to prove the lemma it is enough to show that, 
$$\h^{\sY}_{\sX}(E_{\bullet})(g_{\sX*} \mathcal{G}) = g_{\sY*}(\h^{\sY'}_{\sX'}(g^*_\sY E_{\bullet})(\mathcal{G})) \text{ in }  K_0(\Coh_{\sX} \sY)$$ 
where we ignore the identifications via d\'{e}vissage induced by $i_*$ and $i'_*$. 

Now let us recall two basic facts, firstly for an abelian category $\sC$, there is a natural isomorphism $K_0(\sC) \xrightarrow{\cong} K_0(D^b(\sC))$, which is induced by sending any object of $C \in \sC$ to the complex with $C$ in degree zero and all other terms zero. The inverse of this isomorphism is given by sending any complex $[C_{\bullet}] \mapsto \Sigma_{i} (-1)^i H^i(C_{\bullet})$. Secondly, for  $\sY$ a $\DM$-stack and $\sX \inj \sY$ a closed substack the inclusion induces an equivalence of categories $D^b(\Coh_{\sX} \sY) \to D^b(\Coh \sY)_{\sX}$, where $D^b(\Coh \sY)_{\sX}$ is the full subcategory of $D^b(\Coh \sY)$ of complexes cohomologically supported on $\sX$ (see \cite[Lemma 2.1]{Orlov} in the case of schemes). Hence  
we observe that we have the following natural isomorphisms $$G_0(\sX) \cong K_0(\Coh_{\sX} \sY) \cong K_0(D^b(\Coh_{\sX} \sY)) \cong K_0(D^b(\Coh \sY)_{\sX}).$$
Further, note that for any $C_{\bullet} \in D^b(\Coh \sY)_{\sX}$, we have $\Sigma_{i} (-1)^i [H^i(C_{\bullet})] = [H_+(\Upsilon(C_{\bullet})] - [H_-(\Upsilon(C_{\bullet})]$ in $K_0(\Coh_{\sX}\sY)$. Now essentially from the definition of pushforward at the level of $G$-theory, it follows that for a complex $W_{\bullet} = \Upsilon(C'_{\bullet})$, where $C'_{\bullet} \in D^b(\Coh \sY')_{\sX'}$, we have the following, 
\begin{equation}\label{eq:properlemmafolding}
g_{\sY*}([H_+(W_{\bullet})] - [H_-(W_{\bullet})]) = [H_+(Rg_{\sY*}W_{\bullet})] - [H_-(Rg_{\sY*}W_{\bullet})] \text{  in   } K_0(\Coh_{\sX}\sY). 
\end{equation} 

Now with notations as in  Construction \ref{construction: Homotopyinvarianceconstruction} on $\sY'$ with $A_{\bullet} := g_{\sY}^*E^2_{\bullet} \otimes \Upsilon(\mathcal{G})$ and keeping  rest of the notations as in the construction we can construct $\sA_{\bullet}$ on $\sY' \times \A^1_k$ such that $0^! \sA_{\bullet} \cong j_{0*}C_{\bullet}$
a folding and $1^!\sA_{\bullet} \cong j_{1*}(A_{\bullet}) = j_{1*}(g_{\sY}^*E^2_{\bullet} \otimes \Upsilon(\mathcal{G}))$ the original complex.

Let us denote by $\tilde{g}:=  g_{\sY} \times 1: \sY' \times \A^1_k \to \sY \times \A^1_k$. As $g_{\sY}$  is a proper representable morphism so is $\tilde{g}$. 
We have the following diagram which commutes, 
\begin{equation}\label{eq:lemproper3}
	\xymatrix{
\sY' \ar[r]^-{j'_i} \ar[d]_{g_{\sY}} & \sY' \times_k \A^1_k \ar[d]^{\tilde{g}} \\
\sY  \ar[r]_-{j_i}               & \sY \times_k \A^1_k  ,                  	
}
\end{equation}
where $j_i$ and $j'_i$ correspond to the closed immersion along $\sY \times \{i\}$ and $\sY' \times \{i\}$ respectively.
 As $R\tilde{g}_* : D(\sY' \times_k \A^1_k) \to D(\sY \times_k \A^1_k)$ is a triangulated functor it preserves triangles, thus applying this to \eqref{constr:gysin0} and \eqref{constr:gysin1}  we have the following triangles,
\begin{equation}\label{eqn:Proper0gysin}
	R\tilde{g}_* \sA_{\bullet} \xrightarrow{t} R\tilde{g}_* \sA_{\bullet} \to R\tilde{g}_* 0^!\sA_{\bullet} \xrightarrow{+1},
\end{equation}

\begin{equation}\label{eqn:Proper1gysin}
R\tilde{g}_* \sA_{\bullet} \xrightarrow{(t-1)} R\tilde{g}_* \sA_{\bullet} \to R\tilde{g}_* 1^!\sA_{\bullet} \xrightarrow{+1},
\end{equation}
 in $D(\sY \times_k \A^1_k)$. Note that the cone of the first morphisms in \eqref{eqn:Proper0gysin} and \eqref{eqn:Proper1gysin} coincide with $0^!R\tilde{g}_* \sA_{\bullet} $ and $1^!R\tilde{g}_* \sA_{\bullet} $ respectively. Hence we conclude that, 
 
 \begin{equation}\label{eq:propergysin0}
 0^!(R\tilde{g}_* \sA_{\bullet}) \cong R\tilde{g}_* (0^!\sA_{\bullet}),
 \end{equation}
 \begin{equation}\label{eq:proper1gysin1}
1^!(R\tilde{g}_* \sA_{\bullet}) \cong R\tilde{g}_*( 1^!\sA_{\bullet}).
 \end{equation}
Now note that $0^!([H_+ (R\tilde{g}_*\sA_{\bullet})]- [H_- (R\tilde{g}_* \sA_{\bullet})] ) = 1^!(([H_+ (R\tilde{g}_* \sA_{\bullet})]- [H_- (R\tilde{g}_* \sA_{\bullet})] ))$ in $G_0(\Coh_{\sX}\sY)$. 
 
The proof of the lemma now follows from combining the above observations with \lemref{lem:gysin}. Consider the following, 
 \begin{align*}
 g_{\sY*}([H_+(A_{\bullet})] - [H_-(A_{\bullet})]) &=	g_{\sY*}([H_+(C_\bullet)] - [H_-(C_{\bullet})]) \\
                                                                         &=  [H_+(Rg_{\sY*} C_{\bullet})] - [H_-(Rg_{\sY*} C_{\bullet})] \\
                                                                        &=   0^!([H_+ (R\tilde{g}_*\sA_{\bullet})]- [H_- (R\tilde{g}_* \sA_{\bullet})] )\\
                                                                        & =1^!(([H_+ (R\tilde{g}_* \sA_{\bullet})]- [H_- (R\tilde{g}_* \sA_{\bullet})] ))        \\
                                                                       & = [H_+(Rg_{\sY*}A_{\bullet})] -  [H_+(Rg_{\sY*}A_{\bullet})],                                                               
 \end{align*}
 where the first equality follows from \eqref{eq:constrfoldingandcohomology} and the second equality follows from  \eqref{eq:properlemmafolding} as $C_{\bullet}$ is a folding complex. The third equality follows from the commutativity of  \eqref{eqn:Proper0gysin},  \eqref{eq:propergysin0} and \lemref{lem:gysin}. The second last equality follows from  \lemref{lem:gysin} and the fact that $G_0(\sX \times \A^1_k) \cong G_0(\sX)$(see \lemref{homInv}). The last equality again follows from the commutativity of \eqref{eqn:Proper1gysin} and \eqref{eq:proper1gysin1} and  \lemref{lem:gysin}.

\end{proof}

\begin{remk}
The reader might note that  in \defref{Defn:$h^Y_X$} instead of working with a $2$-periodic complex of vector bundles $E_\bullet$ on $\sY$ which is locally contractible off $\sX$, one could have instead  only imposed the condition that $E_\bullet$ is absolutely acyclic off $\sX$. This remark also applies to the rest of the functoriality statements  of \secref{Sec:FunctorialityProperties}, as the proofs go through as is. 
\end{remk}

	\section{Koszul Complexes} \label{Kos:Comp}

		In \cite[Theorem 4.1]{KiemLi2}, Kiem and Li defined a $K$-theoretic version of the  cosection-localized Gysin map.
		In this section we show that when $E_{\bullet}$ is the Koszul $2$-periodic complex the map $\h^{\sY}_{\sX}(E_{\bullet})$ is equivalent to the cosection localized Gysin map; Theorem \ref{Thm1}.
		We apply the equivalence to prove a comparison of virtual structure sheaves;  \thmref{Comp:Thm}. The reader is advised that what follows is precisely  what is done in \cite[\S~2]{KO1} for the localized Chern character while we are interested in similar properties for the map $\h^{\sY}_{\sX}$ (\defref{Defn:$h^Y_X$}). Most of the statements and their proof rely on the constructions of {\textit{loc.cit}}. In  an effort to  make the paper reasonably self contained we recall the important definitions and constructions and refer the reader \cite{KO1} for further details.

	We recall  the  notion of a Koszul $2$-periodic complex. 
	\begin{defn}[Koszul $2$-periodic complex](\cite[\S~2.2]{KO1})\label{defn:Koszul2periodicComplex}
		
Let $E$ be a vector bundle on a $\DM$-stack $\sY$ and let $\ka \in H^0 (\sY, E^\vee)$, $\kb \in H^0(\sY,  E )$ be sections
such that $\lan \ka , \beta \ran =0$, where the bracket $\lan \ , \ \ran$ denotes an evaluation, or equivalently a pairing.
Let $\{\ka , \kb\}$ denote the following $2$-periodic complex of vector bundles,
\begin{equation}\label{eq:koszul}
\xymatrix{   \oplus _k \bigwedge ^{2k-1} E^{\vee }  \ar@<1ex>[r]^{\wedge \alpha + \iota_\beta  }  
	&  \ar@<1ex>[l]^{\wedge \alpha +  \iota_\beta }  \oplus_k \bigwedge ^{2k} E^{\vee }  }
\end{equation} 
where $\iota _\beta$ denotes the interior product by $\beta$. The convention is that the direct sum of even wedges sits on even degree.
$\{\ka,  \kb  \}$ is called the  {\em Koszul $(2$-periodic) complex associated to the cosection $\ka$ and the section $\kb$}.
	\end{defn}
	 
	 \begin{remk}
	  It follows from \cite[Proposition 2.3.3]{MF}  that the Koszul complex $\{  \ka, \kb  \}$ is locally contractible off   $\sX:=Z (\ka, \kb ):= Z(\ka) \cap Z(\kb )$. Further, from \cite[equation 2.1]{KO1} we have,
	  \begin{equation}\label{duality}
\{\ka,\kb\}= \{\kb^\vee, \ka^\vee\} \otimes \Upsilon(\det E^\vee [rank E] ).
	  \end{equation}
	 \end{remk}

	 We next prove a $K$-theoretic analogue of the Splitting principle proved in \cite[\S~ 2.4]{KO1}. We recall the set-up and notations as in {\it{loc.cit}}.
	 Let $E$, $\alpha$, $\kb$ and $\{\ka,\kb\}$ be as in \defref{defn:Koszul2periodicComplex}. Let $Q$ be a vector bundle on $\sY$ which is a quotient of $E$ and let us denote by $K$ the kernel of the quotient map. We thus have the following exact sequence, 
	 \begin{equation}
	 0 \ra K \xrightarrow{f} E \rightarrow Q \ra 0.
	 \end{equation}
	 Let us assume that the cosection $\alpha$ factors as a cosection of $Q$ which we denote by $\alpha_Q$ and let $\beta_K$ be a section of $K$.
	  Recall from  \cite[\S~2.4.1]{KO1}, we can construct a vector bundle $P$ on $\sY \times  \Spec ~k[t]$ with a cosection $\ka_P$ and a section $\kb_P$ induced by $\ka$ and $\kb$ such that 
	  $P|_{t=0} \cong Q \oplus K$ and $P|_{t =1} \cong E$. With setup as above we  prove the $K$-theoretic splitting principle analogous to \cite[Lemma 2.4]{KO1}.

	\begin{lem}\label{lem:splittingprinciple}
	  	With notations as above let us further assume that $Z(\ka _P, \kb _P ) \subset \sX' \times \mathbb{A}^1_k$ for some closed stack $\sX'$ of $\sY$. 
		Then for  any $\mathcal{G} \in G_0(\sY)$,
	  	\begin{equation}
	  	\h^\sY_{\sX'}   (\{ \ka, \kb \})(\mathcal{G})  =  
		\h^\sY_{\sX'} (\{\ka_Q, 0 \}\ot\{0,\beta_K \}) (\mathcal{G}) .
	  	\end{equation}
	  \end{lem}

  \begin{proof}
  	It follows from \lemref{homInv}.
  \end{proof}

	\subsection{Tautological Koszul complex}\label{taut}
	The goal of this section is to prove \thmref{Thm1},  which is a comparison result between the cosection localized Gysin map  constructed by Kiem and Li (see \cite[Theorem 4.1]{KiemLi2}) and map defined in \defref{Defn:$h^Y_X$}. For the case of the localized Chern character this is proved in \cite[Theorem 2.6]{KO1}, as the strategy of the proof is completely identical we first recall the set-up as in \cite[\S~2.3]{KO1}. 
	
	Let $M$ be a $\DM$-stack and let $F$ be a vector bundle on $M$. Let $\sigma \in H^0(M, F^\vee)$ be a cosection of $F$. Let us denote the total space of $F$ by $|F|$ and $p: |F| \to M$ the projection map. $\sigma$ then  defines a regular function $w_{\sigma}: |F| \to \A^1_k$. Let us denote by $t_F$ the tautological section in $H^0(|F|, p^*F)$ and further  note that $\lan p^*\sigma, t_F \ran = w_{\sigma}$.
	Now we can define a matrix factorization $\{p^*\sigma, t_F\}$ for the regular section $w_{\sigma}$ in a way similar to \eqref{eq:koszul}. 
	Let $Z(w_{\sigma})$ denote the zero locus of $w_{\sigma}$, then  $\{p^*\sigma, t_F\}$  becomes a $2$-periodic complex when restricted to $Z(w_{\sigma})$. 
	 Here, $Z(w_\sigma)$, $p^*F|_{Z(w_\sigma)}$, $p^*\sigma$ and $t_F$ play roles of $\sY$ ,$E$, $\ka$, and $\kb$, respectively, in \defref{defn:Koszul2periodicComplex}  where we change notation to maintain consistency with {\it loc.cit}.
	In the particular case when $\sigma = 0$, then we note that  $\h^{|F|}_{M} (\{0,t_F\}$ coincides with the refined Gysin pull-back by the zero section of $F$ as $\{0,t_F\}$ is the Koszul resolution of $0_*\sO_{M}$.   
	
	Following \cite[Theorem 4.1]{KiemLi}, 
	we consider the blow-up $M'$ of $M$ along $Z(\sigma)$. We consider the notations and set-up as in \cite[\S~2.4.2]{KO1}.
	Let $F'$ and $\sigma'$ denote the pullback of $F$ and $\sigma$ to $M'$, respectively, and let $D$ denote the exceptional divisor.
		The cosection $\sigma' : F' \ra \cO_{M'}$ factors as follows:
		\[ \xymatrix{        F'   \ar[r] \ar[rd]_{\sigma '}  & \cO_{M'} (-D)    \ar[d]_{s_D} \ar[r] & 0 \\
			& \cO_{M'} .  &     } \]  
		Let $K$ be the kernel of a surjective morphism $F' \ra \cO_{M'}(-D)$. Let $b: Z(\sigma ' ) = D \ra Z(\sigma )$ and $p' : |F'| \ra M'$ denote  the natural projection maps. We denote by $g$ the natural map $g: | K |  \hookrightarrow Z(w_{\sigma'}) \ra Z(w_\sigma )$, which is projective, hence in particular note that it is representable. 
		
		In the next lemma we summarize an intermediate construction as in the proof of  \cite[Theorem 4.1]{KiemLi2}.  Using the localization sequence for $G$-theory it follows that we have the following right exact sequence: 
			\begin{equation}\label{eqn:localization}
			G_0(|K|_{Z(\sigma')} |) \to G_0(|K|) \xrightarrow{j^*} G_0( |K| \backslash |F'|_{Z (\sigma')} |) \to 0.
			\end{equation}     
			Further, by construction it follows that $Z(w_\sigma) \backslash |F|_{Z (\ka)} | \cong |K| \backslash |F'|_{Z (\sigma')} |$. For $\mathcal{G} \in G_0(Z(w_{\sigma}))$, let $\mathcal{G}_1$ denote its pull-back to $G_0(Z(w_\sigma) \backslash |F|_{Z (\sigma)} |)$. Under the isomorphism of 
			$Z(w_\sigma) \backslash |F|_{Z (\sigma)} |$ with  $|K| \backslash |F'|_{Z (\sigma')} |$, let $\mathcal{G}_2$ denote its image in $G_0( |K| \backslash |F'|_{Z (\sigma')} |)$. Let $\tilde{\mathcal{G}} \in G_0(|K|)$ be such that $j^*(\tilde{\mathcal{G}}) = \mathcal{G}_2$, which always exists by the localization sequence \eqref{eqn:localization}. 
			
			Consider the localization sequence 
		\begin{align*}
		G_0(|F|_{Z(\sigma)}|) \xrightarrow{\iota_*} G_0(Z(w_\sigma)) \xrightarrow{i^*} G_0(|K| \backslash |F'|_{Z(\sigma')}|) \ra 0.
		\end{align*}
		By  construction, $i^*(\mathcal{G} - g_*\tilde{\mathcal{G}})=0$.
		Hence there is an (not necessarily unique) element $\mathcal{S} \in G_0(|F|_{Z(\sigma)}|)$ such that $\iota_*(\mathcal{S})=\mathcal{G} - g_*\tilde{\mathcal{G}}$.
			
			\begin{lem}\label{lem:const1}
				With notations as above we have the following, 
				\begin{equation}
				\h^{ Z( w_\sigma ) }_{Z(\sigma)}   ( \{ p^* \sigma,  t_F \} ) ( \mathcal{G} ) = \h^{ |F |_{Z(\sigma)}  | }_{Z(\sigma)}  ( \{ 0,  t_{F | _ {Z( \sigma )} } \} )  (\mathcal{S})  +   b_*  ( -[\cO_D( D )] \otimes   \h^{  | K | } _{M'}   (  \{ 0 ,  t_K \}  )   (    \tilde{\mathcal{G}}   ) ).
				\end{equation}
				In particular, the right-hand side is independent of the choice of $\tilde{\mathcal{G}}$ and $\mathcal{S}$.
			\end{lem}
		
		\begin{proof}
		Since $\h^{ |F |_{Z(\sigma)}  | }_{Z(\sigma)}  ( \{ 0,  t_{F | _ {Z( \sigma )} } \} )  (\mathcal{S}) = \h^{ Z( w_\sigma ) }_{Z(\sigma)}   ( \{ p^* \sigma,  t_F \} ) ( \mathcal{G} - g_* \tilde{\mathcal{G}})$ by \lemref{lem:functoriality:Proper}, it is enough to show that $\h^{ Z( w_\sigma ) }_{Z(\sigma)}   ( \{ p^* \sigma,  t_F \} ) (g_* \tilde{\mathcal{G}}) =  b_*  ( -[\cO_D( D )] \otimes   \h^{  | K | } _{M'}   (  \{ 0 ,  t_K \}  )   (    \tilde{\mathcal{G}}   ) )$.
		We obtain
		\begin{align*}
		\h^{ Z( w_\sigma ) }_{Z(\sigma)}   ( \{ p^* \sigma,  t_F \} ) (g_* \tilde{\mathcal{G}}) & =  b_*  (\h^{|K|}_{D} (\{p'^{*}\sigma', t_{F'}\} |_{|K|})(\tilde{\mathcal{G}})) \\
		& = b_*(\h^{|K|}_{D} (\{s_D,0\} \ot \{0, t_K\})(\tilde{\mathcal{G}})) \\
		& = b_*(\h^{M'}_{D} (\{s_D,0\} )  (\h^{|K|}_{M'} (\{0, t_K\})(\tilde{\mathcal{G}}))) \\
		& = b_*  ( -[\cO_D( D )] \otimes   \h^{  | K | } _{M'}   (  \{ 0 ,  t_K \}  )   (    \tilde{\mathcal{G}}   ) )
		\end{align*}
		Here, the first equality is by \lemref{lem:functoriality:Proper}; the second equality is by \lemref{lem:splittingprinciple}; the third equality comes from \lemref{Ring:hom}; and the fourth equality follows from the short exact sequence
		$$0 \ra \cO_{M'}(-D) \xrightarrow{s_D} \cO_{M'} \ra \cO_D \ra 0$$
		and the equation \eqref{duality}.
		\end{proof}
	
	\subsection{Cosection Localization}
	
	Kiem and Li in \cite{KiemLi2} defined the cosection localized Gysin map:
	$0^!_{F, \sigma } : G_0(Z(w_\sigma )) \ra G_0(Z(\sigma ))$. The following theorem is a direct consequence of \defref{Defn:$h^Y_X$}, 
	\lemref{lem:const1} and \cite[Theorem 4.1]{KiemLi2} and is the $K$-theoretic counterpart to \cite[Theorem 2.6]{KO1}.
	\begin{thm}\label{Thm1}
		$0^!_{F, \sigma } = \h^{ Z(w_\sigma)  }   _{  Z(\sigma) } (  \{p^*\sigma , t_F\} )$ as homomorphisms $G_0(Z({w_{\sigma})}) \to G_0(Z(\sigma))$.
	\end{thm}
	
	\begin{proof}
		It follows from directly by comparing \lemref{lem:const1} and the construction of the cosection localized Gysin map; see  \cite[Theorem 4.1]{KiemLi2}, in particular equation $4.9$ of {\textit{loc.cit}}.
	\end{proof}

	In \cite[Theorem 5.1]{KiemLi2}, Kiem and Li prove a comparison result between the cosection localized virtual structure sheaf and the virtual structure sheaf as defined in \cite[\S~2.3]{YPLee}. In \corref{CoVir:Sh} below we state the comparison in our context which is a direct consequence of \thmref{Thm1} and the cycle theoretic analogue  of \cite[Corollary 2.7, Corollary 2.8]{KO1}. The  assumptions are as in \cite[\S~2.5]{KO1} which we briefly recall. Let $A\xrightarrow{d} F$ be a complex of vector bundles on $M$ whose dual gives rise to a perfect obstruction theory relative to a pure dimensional stack $\fM$, let $C$ denote the corresponding cone in $F$ and $\sigma$ a cosection of $F$ such that $\sigma \circ d =0$.
	With respect to the cosection $\sigma$ let  $[  \cO^{\vir}_{M, \sigma}  ] $ denote the cosection-localized virtual structure sheaf (see \cite[Theorem 5.1]{KiemLi2}).\\
	
	\begin{cor}\label{CoVir:Sh} With notations as above the following holds:
		\[ [  \cO^{\vir}_{M, \sigma}  ] := 0^!_{F, \sigma} [\cO_C]  = \h^{ Z(w_\sigma)  }   _{  Z(\sigma) } ( \{ p^* \sigma , t_F \}) ( [\cO_C]  ) . \]
	\end{cor}

	\noindent If $\sigma=0$, $[  \cO^{\vir}_{M, 0 } ]$ in \corref{CoVir:Sh}, then it reduces to the  original definition by Y.-P. Lee \cite[\S~2.3]{YPLee}.

	\section{A Comparison Result}\label{section:com_sh}

	In this section, we assume that the base field is $\mathbb{C}$. The main goal of this section is to prove the $K$-theoretic analogue of 
	\cite[Theorem 3.2]{KO1}. Before we state the main theorem we need to recall the set-up and notations from \cite{KO1}. In  \secref{subsec: SetupNotations} we recall all the notations and basic theorems which would be needed to state the main theorem precisely and \secref{subsec: Proof of Main Theorem} is dedicated to the proof of \thmref{Comp:Thm}. 
	The proof closely follows the strategy as in \cite[Section 3.2]{KO1} and relies on the geometric constructions proved there. We apply Lemma \ref{homInv}, Lemma \ref{Ring:hom} and Theorem \ref{Thm1} to prove Theorem \ref{Comp:Thm}.  The assumptions and notations used are as in {\it{loc.cit}} and we summarize  them below to make the text comprehensible.

	 \subsection{Setup}{\cite[\S~3.1]{KO1}}\label{subsec: SetupNotations}
	 We briefly recall the basic set-up and notations and refer the reader to \cite[\S~3.1]{KO1} for a detailed discussion. \\

	 \paragraph*{\bf Geometric side of a hybrid gauged linear sigma model  $(V_1 \oplus V_2, G, \theta, w)$:}

	 Let $G$ be a reductive linear algebraic group  and let $V_1 := \Spec ~\A^{m}_\C$ and let $V_2: =\Spec~ \A^{n}_\C$ be affine varieties over $\Spec~ \C$. Let us assume that $V_1$ and $V_2$ have a linear $G$-action. Note that $V_1$ and $V_2$  could also be viewed as vector spaces over $\C$ and when confusion does not arise we also denote the underlying vector spaces with the same notation. Let $\theta$ be a character of $G$ such that the semistable and stable points of $V_1$  with respect to $\theta$ coincide, which we denote by $V_1^{ss}(\theta)$. We denote by $E$ the quotient stack $[({V_1^{ss}(\theta) \times_{\C} V_2})/{G}]$, which is naturally  the total space of a vector bundle on the stack $[V_1^{ss}(\theta)/G]$.  Note that $w \in ((\Sym^{\ge 1} V^{\vee}_1) \ot V_2^\vee)^{G}$ induce a regular morphism $ V_1 \to V_2^\vee$ which we denote by $f$, a section $[V_1^{ss}(\theta)/G] \to E^\vee$ which we denote by $s$ and a regular morphism $E \to \A^1_\C$ which is denoted by $\underline{w}$. Let $Z(d\underline{w})$ denote the critical locus of $\underline{w}$ and $Z(s)$ denote the zero locus of $s$. Note  that $Z(d\underline{w}) \cap [V_1^{ss}(\theta)/G]  = Z(s)$, where we further work under the assumption that $Z(d\underline{w})$ is a smooth closed sub-locus in $[V_1^{ss}(\theta)/G]$ of codimension $n$.\\

\paragraph{\bf Moduli Spaces:} 
\begin{enumerate}
	\item Let $\fM _{g, k}(BG, d)$  denote the  moduli space of principal $G$-bundles on genus $g$, $k$-pointed prestable orbi-curves $C$ with degree $d$ such that the associated classifying map $C \ra BG$ is representable.  $\fM _{g, k}(BG, d)$ is a smooth Artin stack (see \cite[\S~2.1]{CKM}, \cite[\S~2.4.5]{CCK}). Let $\pi: \mathfrak{C} \to \fM _{g, k}(BG, d)$ and $\mathcal{P}$ denote the universal curve and the universal $G$-bundle on $\mathfrak{C}$ respectively. For notational convenience we let $\fB: = \fM _{g, k}(BG, d)$.\\
	
	\item Let $Q^{\ke}_{g, k} ( \Zs , d)$ denote the moduli space of genus $g$, $k$-pointed $\ke$-stable quasimaps to $Z(s)$ of degree $d$ which is a separated $\DM$-stack (see \cite[\S~2]{CCK}). Let $X : = Z(s)$, then we let $Q_X^{\ke} := Q^{\ke}_{g, k} ( \Zs , d)$ and   $Q_{V_1}^{\ke} := Q^{\ke}_{g, k} ( V_1/\!\!/G , d)$.\\
	
	\item\cite[\S~3.1.2]{KO1} Let $LGQ^{\ke}_{g, k} (E, d)'$ denote the moduli space of genus $g$, $k$-pointed, degree $d$, $\ke$-stable quasimaps to $V_1/\!\!/G$ with $p$-fields( see \cite{CL, MF,  FJR:GLSM}) which is a separated $\DM$-stack. When confusion does not arise we abbreviate this notation to  $\LGQ '$. \\
	\end{enumerate}

	\paragraph*{\bf  Perfect Obstruction Theories:}
	\begin{enumerate}
			\item Consider the notations as in  \cite[\S3.1.1: Conventions]{KO1}. Let $\sV_i: =  \sP \times_G V_i$ be the  bundles on $\mathfrak{C}$ for $i = 1,2$ and let us denote by $\mathbf{u}$  the universal section. Note here that we use the same abuse of notation as in {\it loc.cit} where $\fC$ and $\sP$ are used to denote the universal curve and universal $G$-bundle on all moduli spaces defined above. \\
			
			\item $Q_X^{\ke}$ has a canonical perfect obstruction theory relative to $\fM _{g, k}(BG, d)$ given by the dual to, 
			\begin{equation} \label{Obst}
				R\pi_*(\mathbf{u}^*(\sP \times_G df^\vee)) : R\pi_*\sV_1 \to R\pi_*\sV_2^{\vee}, 
			\end{equation} 
			see \cite[\S~3, equation 3.2]{KO1}.\\
			
			\item $LGQ'$ has canonical perfect obstruction theory relative to $\fM _{g, k}(BG, d)$ given by, 
			\begin{equation}
				R \pi _* ( \cV_1 \op  \cV_2 \ot \omega_\fC)^\vee, 
			\end{equation}
			where $\omega_\fC$ is a dualizing sheaf on $\fC$ relative to $\LGQ'$. 
			Further, there is cosection $dw_{\LGQ '}: R ^1\pi_* (\cV_1\oplus \cV_2\ot \omega _{\fC}) \ra \cO_{LGQ'} $ (see \cite[\S~3.1.2]{KO1}).\\
	\end{enumerate}

\begin{defn}[$\det R \pi_* \cV$ \& $\chi^{gen}(R \pi_* \cV)$]\label{defn:det}
	Let $\pi: \fC \ra Q^\ke_{g,k}(V_1 /\!\!/ G, d)$ be the universal curve and let $\cV$ be a vector bundle on $\fC$.	As $\pi$ is ample, 
	$\cV$ can always fit into a short exact sequence $0 \ra \cV \ra \mathcal{A} \ra \mathcal{B} \ra 0$ for some $\pi$-acyclic vector bundles $\mathcal{A}$ and $\mathcal{B}$, and further  $\pi_*\mathcal{A}$ and $\pi_*\mathcal{B}$ are vector bundles on $Q^\ke_{g,k}(V_1 /\!\!/ G, d)$. Then we define, 
	\begin{align}\label{Determinant}
\det R \pi_* \cV: &= \det \pi_* \mathcal{A} \ot (\det \pi_* \mathcal{B})^\vee. 
	\end{align}
	$\det R \pi_* \cV$ is called the {\it determinant } of $R\pi_* \sV$.\\
	Further note that $R^i\pi_*\sV = 0$ is zero for $i \neq 0,1$. Let us define, 
	\begin{equation}
\chi^{gen}(R \pi_* \cV) := \dim_{gen} R^0\pi_*\sV - \dim_{gen} R^1\pi_* \sV,
	\end{equation} 
	where $\dim_{gen}$ denotes the generic dimension of the respective coherent sheaves. $\chi^{gen}(R\pi_*\sV)$  is called the {\it  generic virtual rank } of $R\pi_*\sV$.
\end{defn}

	 	 \begin{remk}
	 	 	Note that  \defref{defn:det} is independent of the choice of the $\pi$-acyclic resolution. Suppose our moduli spaces are quasi-projective and  $0 \to A_1 \to B_1 \to 0$ and $0 \to A_2 \to B_2 \to 0$ are two complexes of vector bundles on the moduli space (not on the universal curves!) which are quasi-isomorphic then the determinant as defined in \eqref{Determinant} coincide for both the complexes. In particular note that this is the case for $Q^{\ke}_X$, $Q^{\ke}_{V_1}$ and $\LGQ'$ when $\ke = \infty$.
	 	\end{remk}
	 
 In Lemma \ref{Const:KO}, Lemma \ref{Const:KO3.2.2} and Lemma \ref{Const:KOcosec} below we briefly summarize the constructions of  \cite[Section 3.2]{KO1} that are required for the proof of the main theorem.

	\begin{lem}[See Section 3.2.1 \cite{KO1}] \label{Const:KO}
		
		\begin{enumerate}
			\item[]
			
			\item On some non-empty open substack $\fBo \subset \fB$, there are chain map representations:
			$[ A_1 \xrightarrow{d_{A_1}} B_1]$ of $R \pi_* \cV_1$ and $[Q \xrightarrow{d_Q} P^\vee]$ of $R \pi_* \cV^\vee_2$ respectively.
			
			\item  There is an open substack $U^{\ke} \subset | A_1 |$ such that $Z(d_{A_1}t_{A_1}) \cap U^\ke = Q_{V_1}^{\ke} $.
			
			\item  There are morphisms $\phi_{A_1} : A_1 |_{U^\ke} \ra Q|_{U^\ke}$ 
			and $\phi_{B_1} : B_1|_{Q^\ke_{V_1}} \ra P^\vee|_{Q^\ke_{V_1}}$ (they are defined on different underlying spaces!) such that
			there is a commutative diagram
			\[
			\xymatrix{
				A_1|_{Q^\ke_{V_1}} \ar[r]^-{d_{A_1}} \ar[d]_-{\phi_{A_1} |_{Q_{V_1}^{\ke}} } & B_1|_{Q^\ke_{V_1}} \ar[d]^-{\phi_{B_1}} \\
				Q|_{Q^\ke_{V_1}} \ar[r]_-{d_Q} & P^\vee|_{Q^\ke_{V_1}}
			}
			\]
			which represents  $\eqref{Obst}|_{Q_{V_1}^{\ke}}$ on $Q_{V_1}^{\ke}$. 
		\end{enumerate}
	\end{lem}
	\medskip
	
	\noindent Let $K$ be the kernel of $(\bar{\phi} _{B_1} - d_Q) |_{Q^\ke_X}$.
	By Lemma \ref{Const:KO} (3), 
	$$[ B_1\oplus Q \xrightarrow{\bar{\phi} _{B_1} - d_Q} P^\vee ] |_{Q^\ke_X}$$
	is surjective and $[A_1 |_{Q^\ke_X} \ra K]$ represents the dual of the natural perfect obstruction theory  {\eqref{Obst}} of $Q^\ke_X$ over $\fBo$. Also, by Lemma \ref{Const:KO} (1), $[A_1 \op P \xrightarrow{d_{A_1} \op -d^\vee_Q}  B_1 \op Q^\vee] |_{LGQ'}$ represents the dual of the natural perfect obstruction theory $R \pi_* (\cV_1 \op \cV_2 \ot \omega_\fC) $ of $LGQ'$ over $\fBo$.
	
	Let $U:=U^\ke \times_{\fBo} |P| \times_{\fBo} | Q| \cong | (P \op Q) |_{U^\ke} |$ be the total space of a vector bundle $P \op Q|_{U^\ke}$ over $U^\ke$ 
	where $|_{U^\ke}$ denotes the pullback under the morphism $U^\ke \ra \fB^\circ$ and $\tilde{U}:= U \times_\C \mathbb{A}^1_\C$.
	Let
	$F:= ( B_1 \op Q^\vee \op Q )  |_U$ be a vector bundle on $U$ and $\tilde{F} := F|_{\tilde{U}}$.
	Let $\tilde{p}: \tilde{F} \ra \tilde{U}$ be the projection.
	In \S 3.2.3 \cite{KO1}, the following is constructed:
	
	\begin{lem} \label{Const:KO3.2.2}
	Let $\A^1_\C : =\Spec ~\C[\lambda]$. There exists a section  $\beta: \cO_{\tilde{U}} \ra \tilde{F}$  such that, 
			\begin{align*} Z(\beta) \cong \left\{
				\begin{array}{cl}
					LGQ'  & \text{if } \lambda \neq 0 \in \mathbb{A}^1_\C, \\
					Q^\ke_X \times_{\fB^\circ} |P | \times_{\fB^\circ} |Q| & \text{if } \lambda = 0 \in \mathbb{A}^1_\C.
				\end{array}
				\right.
			\end{align*}
			See the equation (3.10) in \cite{KO1}. Moreover, when $\lambda \neq 0$, the following is a commutative diagram:
			\[
			\xymatrix{
				Z(\beta|_\lambda)  \ar@{^{(}->}[r] \ar[d]^{\cong} & \tilde{U}|_\lambda \cong U \ar[d]^{pr} \\
				LGQ' \ar@{^{(}->}[r] & U^\ke \times_{\fB^\circ} |P|.
			}
			\]
			Here $pr$ is a projection.
			Similarly, when $\lambda=0$, we have a fiber diagram
			\[
			\xymatrix{
				Z(\beta|_0)  \ar@{^{(}->}[r] \ar[d]^{pr} & \tilde{U}|_\lambda \cong U \ar[d]^{pr} \\
				Q^\ke_X \ar@{^{(}->}[r] & U^\ke .
			}
			\]
	\end{lem}
	
	\noindent By Lemma \ref{Const:KO3.2.2}, we have 
	\begin{align}\label{eqn:structuresheaf}
		C_{LGQ'} (U^\ke \times_{\fBo} |P|) \times_{\fBo} |Q| & \cong C_{LGQ' \times_{\fBo} |Q|} U   \cong C_{Z(\beta)}  \tilde{U} |_\lambda, \ \ \lambda \neq 0, \\
		C_{Q^\ke} U^\ke \times_{\fBo} |P| \times_{\fBo} |Q|  & \cong C_{Q^\ke \times_{\fBo} |P| \times_{\fBo} |Q|} U \subset  C_{Z(\beta)}  \tilde{U} |_0 \nonumber
	\end{align}
	where the notation $C_W$ denotes the corresponding cone for a space $W$.
	
	By the construction of $\beta$ in  \cite[equation 3.10]{KO1}, we observe that
	the restriction to $\lambda \in \mathbb{A}^1_\C$ of the dual of a perfect obstruction theory $[T_{\tilde{U} / \fBo \times \mathbb{A}^1_\C} \xrightarrow{d \beta} \tilde{F}]|_{Z(\beta)}$ of $Z(\beta)$ over $\fBo \times \mathbb{A}^1_\C$ is isomorphic to $[A_1 \op P \xrightarrow{d_{A_1} \op -d^\vee_Q}  B_1 \op Q^\vee] \op [Q \xrightarrow{id} Q] |_{LGQ' \times_{\fBo} |Q|}$ when $\lambda \neq 0$.
	When $\lambda =0$, it deforms to $[A_1 |_{Q^\ke_X} \ra K] \op [P \op Q \xrightarrow{-d^\vee_Q \op d_Q} Q^\vee \op P^\vee] |_{Q^\ke_X \times_{\fBo} |P| \times_{\fBo} |Q|} $ (see \cite[\S~2.4.1 and \S~3.2.4]{KO1}).

	\medskip

	\begin{lem} \label{Const:KOcosec}
		A cosection $\sigma : \tilde{F}|_{Z(\beta)} \rightarrow \cO_{Z(\beta)}$ can be constructed such that
		$
		Z(\sigma) \cong Q^\ke_X \times \mathbb{A}^1_\C
		$
	; see the equation (3.10) in \cite{KO1}.
		Moreover, when $\lambda \neq 0$, $\sigma|_\lambda$ coincides with $dw_{LGQ'} \boxplus 0 : (B \op Q^\vee) \op Q |_{Z(\beta)|_\lambda} \ra \cO_{Z(\beta)|_\lambda}$ under the isomorphism above.
		When $\lambda =0$, $\sigma|_0$ deforms to $0 \boxplus \taut : K \op (Q^\vee \op P^\vee) |_{Z(\beta)|_0} \ra \cO_{Z(\beta)|_0}$ under the deformation above. 
		Here, the cosection $\taut$ is explained in the diagram \cite[page 17]{KO1}.
	\end{lem}

	\subsection{Proof of the Main Theorem}\label{subsec: Proof of Main Theorem}
	
	Now we state and prove the main theorem of this section,

	\begin{thm}\label{Comp:Thm}
		In the Grothendieck group of coherent sheaves on $Q^\ke_X:=Q^{\ke}_{g, k} ( Z(s) , d)$, we have
		\begin{align}\label{Vir:Eq2}[ \cO_{Q^{\ke}_X}^{\vir}] = (-1)^{\chi^{gen} ( R\pi_*\cV ^{\vee}_2 ) } \det R \pi_*(\cV_2 \ot \omega_\fC)^\vee  |_{Q_X^{\ke}} \ot  [ \cO_{LGQ',  dw_{LGQ'}}^{\vir} ]
		\end{align}
		where $\chi^{gen} ( \cV ^{\vee}_2 )$ is generic virtual rank  of  $R \pi _* \cV_2 ^{\vee}$ and $\det R \pi_*(\cV_2 \ot \omega_\fC)^\vee$ is the determinant (see \defref{defn:det}).
	\end{thm}
	
	In  \cite[Corollary 3.5]{KO1}, it is shown that $C_{Z(\kb)} \tilde{U} \subset Z(\tilde{p}^* \sigma \circ t_{\tilde{F}})$. 
	Thus, we have the Koszul complex $\{ \tilde{p}^* \sigma, t_{\tilde{F}}\}$ on $C_{Z(\kb)} \tilde{U}$. 
	By Lemma \ref{Const:KOcosec}, $\{ \tilde{p}^* \sigma, t_{\tilde{F}}\}$ is exact off 
	$Q^{\ke}_X  \times \mathbb{A}^1_\C \subset Z(\kb) \subset C_{Z(\kb)} \tilde{U}$.
	For $\kl \in \mathbb{A}^1_\C$, 
	let 
	$p_{\kl}: | F |_{Z(\beta _{\kl} )} | \ra Z(\beta_{\kl} )$
	denote the  projection. 
	In the following Lemma we prove an analogue of \cite[Lemma 3.6]{CKW} for the structure sheaf.

	\begin{lem}\label{lem:structuresheafandgysin}
		Consider the following diagram where $\sX$, $\sY$ are $\DM$-stacks, and all squares are Cartesian
		\begin{equation}
		\xymatrix{
			\sY_{\infty} \ar@{^{(}->}[r]_{j} \ar[d] & \sY|_{\infty} \ar[r] \ar[d] & \sY \ar[r]  \ar[d]         & 0 \ar[d]_{v} \\
			\sX_{\infty}  \ar@{^{(}->}[r]_{i}         & \sX|_{\infty} \ar[r] \ar[d] & \sX  \ar[r] \ar[d]         & \P^1_\C \\
			& \lambda= \infty        \ar[r]          & \P^1_\C {-}\{1\}                       &
		}
		\end{equation}
		Assume further that the following hold:
		\begin{enumerate}
			\item $i$ and $j$ are closed immersions.
			\item The maps $\sX \to \P^1_\C - \{1\}$ and $\sX \to \P^1_\C$ are flat.
			\item The composite map $\sX_{\infty} \to \P^1_\C$ is flat. 
			\item $[\mathcal{O}_{\sX|_{\infty}}] -[i_*\mathcal{O}_{\sX_{\infty}}]$ is supported on $\sY|_{\infty}$. 
		\end{enumerate}
		Then $j_*[\cO_{\sY_{\infty}}] = \infty^![\cO_{\sY}]$ in $G_0(\sY|_{\infty})$.
	\end{lem}
	\begin{proof}
		The proof is similar to  \cite[Lemma 3.6]{CKW} and  follows from the functoriality of the refined Gysin pullback, proper pushforward and the self intersection formula.  Consider the following, 
	
			\begin{align}
			\infty^!([\cO_\sY]) & =^1 \infty^! v^![\cO_\sX] \\
			& =^2 v^! \infty^! [\cO_\sX]   \nonumber\\
			& =^3 v^! [\cO_{\sX|_{\infty}}] \nonumber\\
			& =^4 v^!i_* [\cO_{\sX_{\infty}}]  \nonumber \\ 
			& =^5 j_*v^![\cO_{\sX_{\infty}}] \nonumber \\
			& =^6 j_* [\cO_{\sY_{\infty}}].
			\end{align}   
		As $\sX \to \P^1_k$ is flat, $=^1$ follows from  flat base change \cite[Lemma 1.2]{HR} . $=^2$ follows from the functoriality of the refined Gysin pullback. $=^3$ follows from the same argument as $=^1$ as $\sX \to \P^1_\C-\{1\}$ is flat. Note that the normal bundle $\sY|_{\infty} \to \sX|_{\infty}$ is trivial. Hence $=^4$ follows from  assumption $(3)$ above and the excess intersection formula (see \cite[page 8]{QU}). $=^5$ and $=^6$ follow from the projection formula and the flat base change respectively.   
	\end{proof}
	\begin{lem}\label{lem:Vir:Exp}
		With notations as above we have the following 
		\begin{equation}\label{Vir:Exp}
		\kl ^! \h ^{C_{Z(\kb)} \tilde{U}}_{{Q^{\ke}_X \times \mathbb{A}^1}} \{ \tilde{p}^* \sigma, t_{\tilde{F}}\} ( [\cO_{C_{Z(\beta)} \tilde{U}} ]  ) = \h _{Q^{\ke}_X}^{C_{Z(\kb_\kl)} U  } \{ p^*_{\kl} \sigma, t_F\}  ( [\cO_{ C_{Z(\beta _{\kl})} U}]   )
		\end{equation}
	\end{lem}
	\begin{proof}
		It follows that  $\kl ^! \h ^{C_{Z(\kb)} \tilde{U}}_{{Q^{\ke}_X \times \mathbb{A}^1}} \{ \tilde{p}^* \sigma, t_{\tilde{F}}\} ( [\cO_{C_{Z(\beta)} \tilde{U}} ]  ) $ equals $ \h ^{(C_{Z(\kb)} \tilde{U}) |_\kl}_{{Q^{\ke}_X }}  \{ p^*_{\kl} \sigma, t_F\}   ( \kl^![ \cO_{ C_{Z(\beta)} \tilde{U} } ]   )$ by Lemma \ref{lem:functoriality:pullback}. Let $j_{\lambda}: C_{Z{(\beta_\lambda)}}U \inj C_{Z(\beta)}U|_{\lambda}$ denote the natural map, then from \lemref{lem:structuresheafandgysin} and \cite[Lemma 3.6]{CKW} it follows  that, 
		\begin{equation}\label{eqn: structuresheaf}
		\h ^{(C_{Z(\kb)} \tilde{U}) |_\kl}_{{Q^{\ke}_X }} \{ p^*_{\kl} \sigma, t_F\}   ( \kl^![ \cO_{ C_{Z(\beta)} \tilde{U} } ]   ) = \h _{Q^{\ke}_X }^{(C_{Z(\kb)} \tilde{U})  |_\kl}  \{ p^*_{\kl} \sigma, t_F\}   ( j_{\lambda*}[\cO_{ C_{Z(\beta _{\kl})} U}]   ). 
		\end{equation}
		The claim now follows by applying \lemref{lem:functoriality:Proper} to the right hand side of \eqref{eqn: structuresheaf}. 
	\end{proof}

	We now complete the proof of the main theorem.
		
		\begin{proof}[Proof of \thmref{Comp:Thm}]
	
	Let $p':  C_{LGQ ' } (U^\ke \times_{\fBo} |P|) \subset | (B_1 \oplus Q^\vee )|_{LGQ'} | \ra \LGQ'$ be the projection.
To simplify notations, let $U_P := U^\ke \times_{\fBo} |P|$ and $\cF^1 : = (B_1 \op Q^\vee ) |_{\LGQ '}$.
	Now from \lemref{lem:Vir:Exp}, we have the following, \begin{align*}  \eqref{Vir:Exp}|_{\lambda =1}  & = \h^{C_{LGQ'} U_P  \times_{\fBo} | Q |}_{Q^\ke _X} ( \{ p'^*  dw_{\LGQ '}, t_{\cF ^1}\}\boxtimes _{\fBo} \{ 0, t_Q \} ) ( [\cO_{C_{LGQ'} U_P\ti _{\fBo} | Q | } ] )  \\
	& =  \h^{C_{LGQ'} U_P}_{Q^\ke _X}  ( \{ p'^* dw_{\LGQ'}, t_{\cF ^1}\} ) (   [\cO_{C_{LGQ'} U_P}]  )  \\
	& =  [\cO_{LGQ',dw_{LGQ'}}^{\vir}] . \end{align*} 
	Here the first equality is from Lemma \ref{Const:KOcosec}; the second equality is by Lemma \ref{Ring:hom}; and the third equality follows from Corollary \ref{CoVir:Sh}.
	
	Let $m: C_{Z(\beta _{0})} U \ra |K|$ and $p_0: C_{Z(\beta _{0})} U \ra Z (\beta _0) $  be projections. On  $C_{Z(\beta _{0})} U  $, we deform the complex $\{ p_0^*\sigma , t_F \}$ 
	supported on $Q^{\ke}_X$ to $\{0, m^*t_K \}\ot \{ p_0^*\taut, 0 \}$ supported also on $Q^{\ke}_X$ under the deformation in Lemma \ref{Const:KOcosec}; see the commutative diagram in \cite[page 17]{KO1}.
	By applying splitting principle (\lemref{lem:splittingprinciple}) to this deformation, $\eqref{Vir:Exp}|_{\kl=0}$ becomes the following:
	\begin{align*} &    \h^{ C_{Z(\beta _{0})} U  }_{Q^{\ke}_X } (\{0, m^*t_K \}\ot \{ p_0^*\taut, 0 \})  ( [\cO_{C_{Z(\beta _{0})} U}] )  \\
	= &  \h^{ C_{Z(\beta _{0})} U }_{Q^{\ke}_X } (\{0, m^*t_K \}\ot \Upsilon (\Lambda ^{-\bullet} (P^{\vee} \oplus Q^{\vee}) \ot \Lambda ^{n} (P \oplus Q)[n])) ( [\cO_{C_{Z(\beta _{0})} U}]  ) \\
	= &  (-1)^{\chi (\cV ^{\vee}_2 )}  \Lambda ^{n} (P \oplus Q)|_{Q_X^{\ke}}  \ot ( \h^{| K| } _{Q^{\ke}_X} \{0, t_K \}  ( [ \cO_{C_{Q^{\ke}_X} U^{\ke}}]  ) ) \\
	= &   (-1)^{\chi (\cV ^{\vee}_2 )}  \det R \pi_*(\cV_2 \ot \omega_\fC)  |_{Q_X^{\ke}}   \ot [ \cO_{Q^{\ke}_{g, k} ( \Zs , d)}^{\vir}]
	\end{align*}
	where $\Lambda ^{-\bullet} (P^{\vee} \oplus Q^{\vee})$ is the Koszul complex, and
	$n$ is the rank of $P\oplus Q$. The first equality follows from \eqref{duality} as,
	$$\{ p_0^*\taut, 0 \} = \{ 0, p_0^*\taut^\vee \} \ot \Upsilon( \Lambda ^{n} (P \oplus Q)[n] )= \Upsilon(\Lambda ^{-\bullet} (P^{\vee} \oplus Q^{\vee}) \ot \Lambda ^{n} (P \oplus Q)[n]). $$
	The second equality follows from Lemma \ref{Ring:hom} and the fact that,  $$\h^{ C_{Z(\beta _{0})} U }_{|K| } (\Upsilon (\Lambda ^{-\bullet} (P^{\vee} \oplus Q^{\vee}) )) ([\cO_{C_{Z(\beta _{0})} U}]) = [\cO_{C_{Q^{\ke}_X} U^{\ke}}],$$ by Corollary \ref{CoVir:Sh}.
	Note that the sign $(-1)^{\sX(\sV_2^{\vee})}$ comes from the shifting by $[n]$.
	The third equality follows from the definitions.

	\end{proof}

\end{document}